\documentclass[11pt]{article}

\newcommand{\de}{\,\mathrm{d}}
\newcommand{\e}{\operatorname{e}}

\newcommand{\inc}{\mathrm{inc}}

\newcommand{\real}{\mathrm{Re}\,}

\newcommand{\imag}{\mathrm{Im}\,}

\newcommand{\R}{\mathbb{R}}
\newcommand{\C}{\mathbb{C}}
\newcommand{\N}{\mathbb{N}}

\usepackage{multicol}
\usepackage{xcolor}
\definecolor{capa}{RGB}{50, 90, 160}
\definecolor{delftblue}{RGB}{0,61,165}
\usepackage[colorlinks=true,
  linkcolor=delftblue,
  urlcolor=delftblue,
  citecolor=delftblue]{hyperref}

\newcommand{\rank}[1]{\mathrm{rank}\left(#1\right)}

\newcommand{\Mtri}{\mathsf{M}^{\mathrm{tri}}}

\usepackage{enumitem}
\usepackage{multirow}
\usepackage{amsmath}
\usepackage{amsfonts}
\usepackage{amssymb}
\usepackage{graphicx}
\usepackage{caption}
\usepackage{subcaption}
\usepackage{mathrsfs}
\usepackage{mathtools}
\usepackage{upgreek}
\usepackage{amsthm}
\usepackage{booktabs}
\usepackage{authblk}
\usepackage{cite}
\usepackage{url}
\usepackage[normalem]{ulem}

\newtheorem{theorem}{Theorem}[section]
\newtheorem{lemma}[theorem]{Lemma}
\newtheorem{proposition}[theorem]{Proposition}

\newtheorem{remark}[theorem]{Remark}

\title{Fast high-order solvers for the Lippmann--Schwinger equation in piecewise-smooth heterogeneous media}

\author[2]{Thomas G. Anderson \thanks{\href{mailto:thomas.anderson@rice.edu}{thomas.anderson@rice.edu}}}
\author[1]{Juan Burbano-Gallegos\thanks{\href{mailto:j.s.burbano@utwente.nl}{j.s.burbano@utwente.nl}}}
\author[3]{Luiz M. Faria\thanks{\href{mailto:luiz.faria@ensta.fr}{luiz.faria@ensta.fr}}}
\author[1]{Carlos P\'erez-Arancibia\thanks{\href{mailto:c.a.perezarancibia@utwente.nl}{c.a.perezarancibia@utwente.nl}}}
\affil[1]{\small{Department of Applied Mathematics, University of Twente, Enschede, the Netherlands}}
\affil[2]{\small{Department of Computational Applied Mathematics \& Operations Research, Rice University, Houston, TX, USA}}
\affil[3]{\small{POEMS, CNRS, Inria, ENSTA Paris, Institut Polytechnique de Paris, France}}

\date{\today}

\graphicspath{%
  {figures/}%
}

\begin{document}
\maketitle

\begin{abstract}
    This article presents a fast, high-order Nystr\"{o}m solver for the two-dimensional Lippmann--Schwinger equation arising from time-harmonic scattering by penetrable, piecewise-smooth heterogeneous media. Relying on high-order accurate evaluation of the Newtonian potential on unstructured grids that are adapted to interfaces of discontinuity in a given heterogeneous media, the methodology achieves high-order accurate approximation of the solution to the Lippmann--Schwinger equation using existing fast algorithms such as the fast multipole method. As an iterative method the solver exhibits rapid convergence when coupled to a preconditioning strategy that exploits a class of structured-grid solvers---a class of fast solution methods which has achieved considerable success and attention, including quasi-linear time and memory complexity for setup and application as well as nearly-constant iteration counts, but which has long been intrinsically limited in accuracy. The preconditioning strategy itself couples the Nystr\"{o}m discretization---given by high-order quadrature nodes over a (curved) unstructured mesh conforming to the integration domain, a domain coinciding with the support of the spatially varying contrast---to a uniform Cartesian grid underlying the fast preconditioner by use of a pair of transfer operators. The resulting preconditioner yields a solver that inherits the frequency-robust behavior of its Cartesian counterpart without sacrificing the geometric flexibility and high-order accuracy of the unstructured discretization. We prove that invertibility of the proposed preconditioner holds under explicit conditions on the mesh sizes and on the Cartesian preconditioner. Numerical experiments demonstrate that the preconditioned system requires significantly fewer GMRES iterations than its unpreconditioned counterpart, with iteration counts almost independent of both mesh size and wavenumber. The paper concludes with numerical results illustrating the method's robustness for inhomogeneities described by piecewise-smooth refractive indices with jump discontinuities across interfaces.
\end{abstract}

\section{Introduction}
\label{sec:introduction}

Integral equation formulations provide a powerful framework for time-harmonic wave scattering by inhomogeneous media where the wavenumber (resp. the refractive index) differs from the background wavenumber $k$ (resp. the background refractive index). For inhomogeneous penetrable scatterers with compact support, the Lippmann--Schwinger equation yields a volume integral formulation that enforces the Sommerfeld radiation condition implicitly and is naturally suited to fast convolution-based solvers. In practice, however, attaining both high-order accuracy and frequency-robust iterative convergence is difficult when the refractive index is only piecewise smooth, exhibiting jump discontinuities across material interfaces.

The numerical solution of the Lippmann--Schwinger equation has received sustained attention over the past three decades, and the existing literature can be broadly organized around the two challenges just mentioned: (a) the accurate evaluation of the weakly singular volume potential over a general bounded domain and (b) the efficient solution of the resulting dense linear systems. Regarding the first challenge, a categorical distinction that is central to the present work must be drawn between high-order methods that assume a \emph{globally smooth} compactly supported contrast (i.e., the deviation of the refractive index from its uniform background value, and the quantity that enters the volume integral operator), and the comparatively few that retain their accuracy when the contrast is only \emph{piecewise smooth}, exhibiting jump discontinuities across material interfaces. The latter setting plays a relevant role in applications. Refractive-index profiles with sharp material interfaces are the norm in man-made structures, such as the dielectric devices encountered in nanophotonics and (integrated) optics, where the abrupt index changes between distinct materials are precisely what give the device its function. While such structures can often be modeled by piecewise-constant refractive indices, for which boundary integral equations are available, applications such as graded-index optics~\cite{gomezreino2002gradient} and devices carrying fabrication-induced gradients require a piecewise-smooth index that varies continuously within each material subdomain.

Most existing high-order solvers fall in the first category. They heavily exploit the convolutional structure of the free-space Green's function on uniform grids, both for designing high-order quadrature rules or for efficiency so that the volume potential can be applied in $\mathcal{O}(N\log N)$ operations via the fast Fourier transform (FFT), with $N$ denoting the number of grid points discretizing the (rectangular) computational domain. Periodization-based solvers of optimal accuracy order were introduced in~\cite{vainikko2000fast}, including a super-algebraic high-order trigonometric collocation scheme for globally smooth scatterers (see also the monograph~\cite{saranen2001periodic}), while FFT-accelerated high-order solvers in two and three dimensions were developed in~\cite{bruno2004efficient,hyde2005fast} with a rigorous convergence analysis given in~\cite{bruno2005higher}. Related Fourier-based techniques include the bandlimited solver of~\cite{andersson2005fast} as well as the fast free-space convolution schemes of~\cite{beylkin2009fast} and~\cite{vico2016fast}, the latter, as~\cite{vainikko2000fast}, achieving superalgebraic convergence for globally smooth compactly supported contrasts. In the Nystr\"om context, corrected quadrature rules on uniform grids that retain high-order accuracy despite the kernel singularity were constructed in~\cite{duan2009high}. These quadratures underlie several of the solvers mentioned below, as well as the Cartesian-grid discretization used by the preconditioner proposed in this paper. What all of these uniform-grid schemes have in common is that their high-order accuracy rests on the smoothness of the volume potential input density, and hence of the refractive index: when the refractive index jumps across an interface that does not conform to the grid, the integrand loses regularity and the observed accuracy collapses to low order. Adaptivity offers only a partial remedy: the adaptive, high-order quadtree discretization of~\cite{Greengard:16}, for instance, resolves smoothly varying refractive indices to high order through local refinement, but its accuracy
degrades across a genuine discontinuity, which refinement toward the interface cannot restore.

The second category---genuine high-order accuracy for a non-smooth or discontinuous refractive index---is far less explored, and has been achieved only by discretizations that conform to the material interface. A high-order Nystr\"om scheme on decomposed domains that retains accuracy for penetrable media possessing discontinuities across globally smooth interfaces---similar to the problem class of this paper excepting that the present work does not require a globally-smooth boundary---was proposed in~\cite{anand2016efficient} and subsequently extended in~\cite{pandey2019improved} and which relies on augmenting fast, low-order convolution solvers with a specialized boundary integrator to recover high-order convergence across the interface at an $\mathcal{O}(N\log N)$ per-iteration cost. The polynomial-based volume density interpolation method (VDIM)~\cite{anderson2024fast}, which underlies the volume-potential evaluation used here, belongs to this second category as well, providing high-order accuracy on (curved) triangulated domains together with uniformly accurate evaluation up to and across domain boundaries. In contrast to the decomposed-domain schemes of~\cite{anand2016efficient,pandey2019improved}---which rely on overlapping coordinate charts, partitions of unity, a uniform Cartesian grid, and problem-specific changes of variables to analytically resolve the kernel singularity---VDIM operates directly on a standard (possibly adaptive, curved) triangular mesh and regularizes the singularity through Green's third identity, delivering comparable high-order accuracy with greater geometric flexibility and a simpler, essentially kernel-independent implementation (indeed, recently,~\cite{anderson2026general} extended VDIM to a broad class of three-dimensional volume integral operators). As discussed below, however, none of these interface-resolving high-order schemes have been paired with frequency-robust preconditioning that keeps the number of Krylov-subspace (GMRES) iterations required to solve the resulting linear system bounded at high frequency. In fact, they frequently exhibit strong growth in required iteration counts even for fixed frequency as the discretization is refined.

A second, complementary line of research concerns fast \emph{direct} solvers (either for Helmholtz equation discretizations or integral formulations), which exploit low-rank approximations of the off-diagonal blocks of the discretized operator to build a compressed approximate inverse (see~\cite{greengard2009fast} for an early survey and~\cite{martinsson2019fast} for a comprehensive treatment). For the two-dimensional Lippmann--Schwinger equation, the first fast direct algorithm was proposed in~\cite{chen2002fast}, with subsequent high-order developments in~\cite{aguilar2004high} and, more recently, in~\cite{gopal2022accelerated}, whose Hierarchically Block Separable solver constructs an approximate inverse in $\mathcal{O}(N^{3/2})$ operations on top of the tenth-order discretization of~\cite{duan2009high}, and proves particularly effective when a low-accuracy inverse is employed as a GMRES preconditioner. In a similar spirit, albeit based on a differential rather than an integral volumetric formulation, the spectrally accurate direct solver of~\cite{gillman2015spectrally} merges impedance-to-impedance maps over a quadtree of boxes and couples the result to a boundary integral equation; see also the hybrid direct/iterative solver of~\cite{bruno2023direct}. Direct solvers are especially attractive when many incident fields must be processed, as occurs, for instance, in inverse medium scattering via recursive linearization~\cite{borges2017high}, though their factorization cost and memory footprint grow significantly in the high-frequency regime. As already mentioned, all such fast direct solvers sharing the structured-grid assumption, even adaptive ones~\cite{Greengard:16}, revert to low-order accuracy in the presence of arbitrary interfaces in the refractive index.

For large-scale and high-frequency problems, this work is motivated by the observation that, in part because of their limited memory demands, iterative methods remain the standard approach---with fast operator application variously provided by the FFT on uniform grids~\cite{vainikko2000fast,saranen2001periodic} or by the fast multipole method~\cite{greengard1987fast,rokhlin1990rapid} in more general settings. Their main drawback is that Krylov iteration counts for the Lippmann--Schwinger equation grow rapidly with the wavenumber and the contrast, which has motivated a substantial body of work on frequency-robust preconditioners and on convergent iterative reformulations. The convergent Born series of~\cite{osnabrugge2016convergent}, for instance, offers convergence guarantees for arbitrarily large media at the expense of an artificial absorption parameter. The sparsifying preconditioner of~\cite{ying2015sparsifying} transforms the dense discretized system into a nearly sparse one resembling a finite-difference discretization of the Helmholtz equation, whose approximate inverse yields nearly frequency-independent GMRES iteration counts. Building on this construction, the sweeping-type preconditioner of~\cite{zepeda2016fast} applies a domain-decomposition sweep to the sparsified Lippmann--Schwinger system through alternating bidirectional sweeps. The subsequent sparsify-and-sweep preconditioner of~\cite{liu2018sparsify} combines the sparsifying construction directly with a moving-PML sweeping factorization, further reducing the iteration count to a handful in both two and three dimensions. Crucially, these frequency-robust solvers are, without exception, designed for and intrinsically tied to uniform Cartesian discretizations.

The picture that emerges from this literature is a tension, in the piecewise-smooth setting, between high-order accuracy and frequency-robust iterative convergence---that is, GMRES iteration counts that stay bounded as the wavenumber grows. Uniform-grid discretizations admit fast frequency-robust preconditioners, but lose high-order accuracy in the presence of curved material interfaces that do not conform to the grid. Fitted-mesh high-order Nystr\"om discretizations resolve such interfaces naturally, but are not directly compatible with structured-grid preconditioners. This paper introduces a solver that leverages the best aspects of both approaches. To the best of our knowledge, it is the first to combine high-order accuracy for piecewise-smooth, interface-conforming contrasts with the frequency-robust, near-linear-cost iterative solution demanded by high-frequency scattering, thereby providing the first fast, high-order Lippmann--Schwinger solver in two dimensions for high-frequency scattering by piecewise-smooth heterogeneous media.

The proposed method uses a fitted triangular mesh with Vioreanu--Rokhlin (VR) quadrature~\cite{Vioreanu:14} to resolve the support of the contrast and evaluate the volume potential to high-order accuracy via VDIM~\cite{anderson2024fast}. The preconditioner, in turn, acts on a uniform Cartesian grid over a rectangular domain enclosing the contrast support, and transfer operators couple the two discretizations. The Cartesian-to-quadrature map is a bilinear prolongation. For the quadrature-to-Cartesian map, we propose and analyze two constant-preserving constructions---a local $L^2$ projection onto piecewise-linear functions and a renormalized transpose of the Cartesian-to-quadrature prolongation---both of which render the preconditioner invertible under certain conditions. We adopt the renormalized transpose, a construction standard in geometric multigrid~\cite{briggs2000multigrid}, in all scattering experiments, as it yields consistently lower GMRES iteration counts that remain robust as the frequency and contrast increase.

This strategy of preconditioning an unstructured-grid discretization with a structured one is related to the auxiliary space framework~\cite{xu1996auxiliary,grasedyck2016nearly}, though that theory assumes a symmetric positive-definite system with invertible smoothers~\cite[(2.2)]{xu1996auxiliary}. These assumptions do not hold for the Lippmann--Schwinger system, where the direct analogue of a smoother in our preconditioner is itself non-invertible (cf.~\eqref{eq:preconditioner_split}), consistent with the well-known ineffectiveness of multigrid for such problems~\cite{bramble1988analysis}. Closer still is~\cite{huang2024grid}, which uses similar transfer operators for a grid-overlay finite-difference discretization of fractional Laplacian problems on unstructured meshes---though there the resulting preconditioner is always non-invertible in the high-order setting we consider (Section~\ref{sec:preconditioner}).

The main contributions of this work are fourfold. First, a high-order Nystr\"om solver for the two-dimensional Lippmann--Schwinger equation in piecewise-smooth media. Second, a new preconditioning strategy that couples the fitted discretization to a Cartesian one via transfer operators, embedding any existing fast preconditioner for the uniform Cartesian discretization---the sparsifying preconditioner~\cite{ying2015sparsifying} in this work. Third, a theoretical invertibility result for the resulting preconditioner under conditions on the mesh refinement ratio and on the sparsifying preconditioner. Fourth, extensive numerical evidence demonstrating that the preconditioned system exhibits significantly reduced GMRES iteration counts while remaining stable with respect to mesh size and wavenumber.

The remainder of the paper is organized as follows. Section~\ref{sec:problem_setup} introduces the scattering problem and the Lippmann--Schwinger formulation. Section~\ref{sec:Nystrom_discretization} describes the quadrature-node Nystr\"om discretization and the VDIM-based volume potential evaluation. Section~\ref{sec:preconditioner} presents the preconditioner construction, and Section~\ref{sec:invertibility} establishes its invertibility. Section~\ref{sec:numerical_examples} presents numerical results illustrating accuracy and preconditioner performance and we conclude in Section~\ref{sec:conclusion}.

\section{Problem setup}
\label{sec:problem_setup}

In this work we consider the scattering problem arising from the interaction of a time-harmonic incident field $u^{\inc}(x)$, a free-space solution to the Helmholtz equation with frequency $k$ such as a plane-wave, with a penetrable heterogeneous obstacle. The medium is characterized by the \emph{refractive index}
\begin{equation}\label{eq:refractive_index_def}
  n(x) = n_1(x) + i\frac{n_2(x)}{k}, \qquad x \in \R^2,
\end{equation}
where $k > 0$ is the (free-space) wavenumber, $n_1=\real n > 0$, and $n_2=\imag n \geq 0$. We  assume throughout that the refractive index $n:\R^2\to\C$ is piecewise smooth and that the \emph{contrast}, defined as
\begin{equation}\label{eq:contrast_def}
  m(x):= 1 - n(x),\qquad x\in\R^2,
\end{equation}
satisfies $\|m\|_{L^\infty(\mathbb{R}^2)} < \infty$ and is compactly supported, i.e., $n \equiv 1$ (homogeneous background) outside a bounded domain $\Omega$; here $\overline\Omega = \mathrm{supp}\, m$, with $\Omega\subset\R^2$ open and bounded with piecewise smooth (Lipschitz) boundary $\Gamma$. The sought total field $u : \mathbb{R}^2 \to \mathbb{C}$ then satisfies
\begin{subequations}\begin{align}
  \Delta u + k^2 n(x)\, u = 0 &\quad \text{in } \mathbb{R}^2, \label{eq:Helmholtz_eq}\\
  u = u^{\inc} + u^s&, \label{eq:total_field_def}\\
  \lim_{|x| \to \infty} |x|^{1/2}\bigl(\frac{\partial u^s}{\partial |x|} - ik&u^s\bigr) = 0,
  \label{eq:sommerfeld_rc_def}
\end{align}\label{eq:bvp}\end{subequations}
where as usual the Sommerfeld radiation condition~\eqref{eq:sommerfeld_rc_def} is enforced for the scattered field $u^s := u - u^{\inc}$.

Under the aforementioned conditions on the contrast $m$, it follows from Green's representation formula (see, e.g.,~\cite{COLTON:2012,costabel2015spectrum}) that $u$ satisfies the \emph{Lippmann--Schwinger equation}
\begin{equation}\label{eq:Lippmann-Schwinger}
  u(x) + k^2 \int_{\Omega} G(x,y)\,m(y)\,u(y)\,\de y = u^{\inc}(x),
  \quad x \in \mathbb{R}^2,
\end{equation}
where $G(x,y) := \frac{i}{4}H_0^{(1)}(k|x-y|)$, $x\neq y$, is the free-space Green's function
of the Helmholtz equation in $\R^2$. As is well known~\cite{COLTON:2012}, the conditions $n_1 > 0$ and $n_2\geq 0$ render~\eqref{eq:Lippmann-Schwinger} well-posed via the Fredholm alternative, admitting a unique solution $u\in H^1_{\rm loc}(\R^2)$. Interestingly, well-posedness of~\eqref{eq:Lippmann-Schwinger} holds for a considerably broader class of regions $\Omega$, including those with fractal set boundaries~\cite{bannister2026acoustic}.

Since the unknown $u$ enters the integral in~\eqref{eq:Lippmann-Schwinger} only through its values on $\overline{\Omega}=\operatorname{supp}m$, restricting the evaluation point $x$ to $\Omega$ already yields a closed integral equation for $u|_\Omega$,
\begin{equation}\label{eq:Lippmann-Schwinger_Omega}
  u(x) + k^2 \int_{\Omega} G(x,y)\,m(y)\,u(y)\,\de y = u^{\inc}(x), \quad x \in \Omega .
\end{equation}
Equation~\ref{eq:Lippmann-Schwinger_Omega} above is the equation we discretize (Section~\ref{sec:Nystrom_discretization}). Once $u|_\Omega$ has been obtained, the total field at any point $x\in\R^2$ is recovered via the representation formula
\begin{equation}\label{eq:representation_formula}
  u(x) = u^{\inc}(x) - k^2 \int_{\Omega} G(x,y)\,m(y)\,u(y)\,\de y,
  \quad x \in \R^2 .
\end{equation}

A similar equation may be posed on any bounded domain containing $\Omega$. In particular, on a rectangle $D\subset\R^2$ with $\overline{\Omega}\subset D$, it holds that
\begin{equation}\label{eq:Lippmann-Schwinger_small}
  u(x) + k^2 \int_{\Omega} G(x,y)\,m(y)\,u(y)\,\de y = u^{\inc}(x), \quad x \in D ,
\end{equation}
where on $D\setminus\Omega$ equation~\eqref{eq:Lippmann-Schwinger_small} coincides with the representation formula~\eqref{eq:representation_formula}. This enclosing rectangle is the domain over which the structured Cartesian grid underlying the preconditioner of Section~\ref{sec:preconditioner} is defined.

\section{Nystr\"om discretization}\label{sec:Nystrom_discretization}

We discretize~\eqref{eq:Lippmann-Schwinger_Omega} via a Nystr\"{o}m method whose collocation points are the high-order quadrature nodes of a fitted triangular mesh over $\Omega$. Because the unknown enters~\eqref{eq:Lippmann-Schwinger_Omega} only through its values on $\Omega$, these interior nodes alone determine the discrete solution. The field at any other point, in particular at exterior evaluation points, is recovered in a single post-processing step from the representation formula~\eqref{eq:representation_formula}. 

The interior quadrature nodes in $\Omega$ allow for the accurate numerical integration of the volume potential
\begin{equation}\label{eq:vol_pot}
  \mathcal{V}_\Omega[f](x) := \int_{\Omega} G(x,y)f(y)\de y, \quad x \in \R^2,
\end{equation}
which appears in~\eqref{eq:Lippmann-Schwinger_Omega} with density $f = m\cdot u$. Since $m$ is only piecewise smooth on $\Omega$, we decompose $\Omega$ into subdomains on which $m$ is smooth. In more detail, $\Omega$ admits a decomposition into a finite collection of disjoint, open, bounded sets $\{\Omega_\ell\}_{\ell=1}^{N_d}$ with piecewise smooth boundaries $\Gamma_\ell = \partial\Omega_\ell$, such that $m|_{\Omega_\ell}$
is smooth for each $\ell \in \{1,\ldots,N_d\}$, and
$\overline\Omega = \overline{\bigcup_{\ell=1}^{N_d} \Omega_\ell}$.
This decomposition is assumed to be \emph{minimal}, in the sense that $m$ does not extend smoothly across any interface shared by two adjacent subdomains. Equivalently, no such pair can be merged while keeping $m$ smooth on their union.

The volume potential~\eqref{eq:vol_pot} is then decomposed as
\begin{equation}\label{eq:vol_pot_decomp}
  \mathcal{V}_\Omega[f](x) = \sum_{\ell=1}^{N_d} \mathcal{V}_{\Omega_\ell}[f](x),
  \quad x \in \R^2,
\end{equation}
and each term is treated independently via VDIM~\cite{anderson2024fast}. The VDIM requires a triangulation of each subdomain, so for each $\Omega_\ell$ we introduce the conforming, curved triangular mesh $\mathcal{T}^{(\ell)}$ and define the global mesh $\mathcal{T} := \bigcup_{\ell=1}^{N_d} \mathcal{T}^{(\ell)}$. By construction, the global mesh $\mathcal{T}$ conforms to the discontinuities of the contrast in the sense that every interface $\partial\Omega_\ell$ is resolved by mesh edges, so that each triangle lies within a single subdomain and $m$ is smooth on it. Following~\cite{anderson2024fast}, we assume throughout that $\mathcal{T}$ belongs to a family of shape-regular and quasi-uniform meshes parametrized by the mesh size $h:=\max_{\tau \in \mathcal{T}} \max_{x,y\in\tau}|x-y|$, i.e., the maximum diameter of any triangle in $\mathcal{T}$.

In a nutshell, VDIM works as follows. For a given target point $x \in \Omega_\ell$, let $\tau = \tau(x) \in \mathcal{T}$ denote the triangle containing $x$. VDIM constructs a degree-$p$ polynomial interpolant $F^{(\tau)}$ of the density $f$ on~$\tau$, together with a polynomial particular solution $\Phi^{(\tau)}$ satisfying $(\Delta + k^2)\Phi^{(\tau)} = F^{(\tau)}$ in $\R^2$. These two target-point-dependent
objects regularize the singular Green's function kernel via Green's third identity,
recasting the volume potential as
\begin{equation}\label{eq:DIM_decomp}
  \mathcal{V}_{\Omega_\ell}[f](x)
  = \mathcal{V}_{\Omega_\ell}[f - F^{(\tau)}](x) + \mathcal{B}_{\Omega_\ell}[F^{(\tau)}](x),
\end{equation}
where $\mathcal{B}_{\Omega_\ell}[F^{(\tau)}](x)$ collects a boundary integral
correction---consisting of single- and double-layer potentials along
$\partial\Omega_\ell$ with polynomial densities---and a local evaluation of the
polynomial PDE solution $\Phi^{(\tau)}(x)$. Since $f - F^{(\tau)}$ is small
sufficiently close to $x \in \tau$ (indeed, it is a quantity of
$\mathcal{O}(h^{p+1})$ there~\cite{anderson2024fast}), the regularized
integrand in $\mathcal{V}_{\Omega_\ell}[f-F^{(\tau)}]$ is smooth and can be
evaluated accurately by standard high-order quadrature over this part of the triangulation.
Further, the error analysis extends guarantees of high-order accurate quadrature
across the entire triangulation $\mathcal{T}$ using this regularization.
The boundary correction $\mathcal{B}_{\Omega_\ell}[F^{(\tau)}](x)$ is in turn
evaluated by boundary quadrature over $\partial\Omega_\ell$: a standard rule when
$x$ lies far from $\partial\Omega_\ell$, and a specialized near-singular
layer-potential evaluation when $x$ lies close to it. For the latter we use the
general-purpose density interpolation method of~\cite{faria2021general}, which together with VDIM
we use as implemented in the open-source \texttt{Inti.jl}~\cite{Inti} package.

Numerical integration is performed using a VR
quadrature rule~\cite{Vioreanu:14} on each triangle. The rule is determined by its
\emph{interpolation degree} $p\in\N$, the maximum degree interpolant possible with the rule, and contains
$n_q = (p+1)(p+2)/2 = \mathcal{O}(p^2)$ strictly interior nodes per element; on a given element $\tau$ we write $\{y_q^{(\tau)},\omega_q^{(\tau)}\}_{q=1}^{n_q}$ for these nodes and their weights. These nodes serve
dually as a degree-$p$ interpolation scheme and as a quadrature rule that
integrates every polynomial of total degree $\leq d_{\mathrm{ex}}(p)$ exactly, where
the \emph{degree of exactness} $d_{\mathrm{ex}}(p) > p$ is tabulated
in~\cite[Table~5.1]{Vioreanu:14}; in particular, $d_{\mathrm{ex}}(2) = 4$. Collecting all nodes across $\mathcal{T}$, we obtain the global quadrature node set $\mathcal{Q} = \{y_j\}_{j=1}^{N_Q}$ (see Figure~\ref{fig:domains_and_mesh}) with corresponding weights
$\{\omega_j\}_{j=1}^{N_Q}$, related to the per-element weights by
$\omega_j = \omega_q^{(\tau)}$ whenever the global node $y_j$ is the local node $y_q^{(\tau)}$
of element $\tau$. Here $N_Q = n_q N_T$, with $N_T := |\mathcal{T}|$ the total number of triangles.

\begin{figure}[ht]
    \centering
    \includegraphics[width=0.4\linewidth]{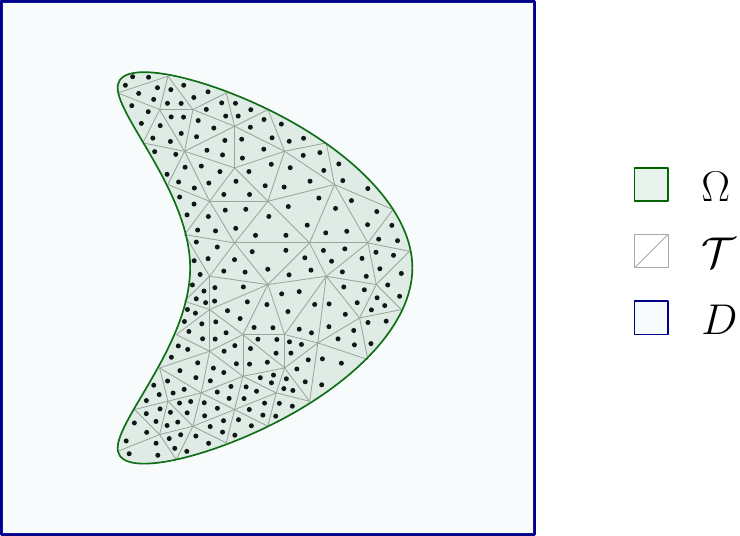}
    \caption{Contrast support domain $\Omega$, embedded in the rectangular domain $D$, together with the fitted triangular mesh $\mathcal T_h$ and its quadrature nodes ($n_q=3$).}
    \label{fig:domains_and_mesh}
\end{figure}

Introducing the vector of density values
$\mathbf{f} \in \mathbb{C}^{N_Q}$, with $(\mathbf{f})_j = f(y_j)$, the
VDIM approximation of the full volume potential defines a linear map
\begin{equation}\label{eq:DIM_map}
  \mathbf{f} \;\longmapsto\; \widetilde{\mathcal{V}}_\Omega[\mathbf{f}](x),
  \qquad x \in \R^2,
\end{equation}
which provides a high-order accurate approximation of $\mathcal{V}_\Omega[f](x)$
everywhere in $\R^2$; for a VR rule of interpolation degree $p$, it generally converges with errors of $\mathcal{O}(h^{p + 3}|\log h|)$ under $h$-refinement. 

The high-order VDIM potential map~\eqref{eq:DIM_map} is applied in a matrix-free fashion using a fast multipole method (FMM) or,
alternatively, through $\mathcal{H}$-matrix compression, so that each application of
$\widetilde{\mathcal{V}}_\Omega$ has quasi-linear rather than quadratic complexity in the
number of unknowns. Even with this fast apply, however, the dense linear system obtained
by collocating the equation requires, in the absence of preconditioning, a number of
GMRES iterations that grows rapidly with the wavenumber and the contrast, making its
iterative solution prohibitive even for moderate problem sizes. This motivates the
use of a structured discretization with an associated fast solver to precondition the VDIM map.

The Nystr\"{o}m method seeks the discrete solution vector $\mathbf{u}\in\mathbb{C}^{N_Q}$ with $(\mathbf{u})_j\approx u(y_j)$ at the quadrature nodes. Collocating~\eqref{eq:Lippmann-Schwinger_Omega} at the $N_Q$ quadrature nodes and replacing the volume potential by its VDIM approximation~\eqref{eq:DIM_map} yields the $N_Q\times N_Q$ linear system
\begin{equation}\label{eq:linear_system}
(\mathbf{u})_j + k^2\,\widetilde{\mathcal{V}}_\Omega[\mathsf{M}\mathbf{u}](y_j)
= u^{\mathrm{inc}}(y_j), \qquad j\in\{ 1, \ldots, N_Q\},
\end{equation}
where $\mathsf{M} \in \mathbb{C}^{N_Q \times N_Q}$ is the diagonal matrix with
$(\mathsf{M})_{jj} = m(y_j)$. In matrix form,
\begin{equation}\label{eq:nystrom_linear_system}
\mathsf{L}\mathbf{u} = \mathbf{b},
\end{equation}
where 
\begin{equation}\label{eq:linear_system_matrix}
\mathsf{L} := \mathsf{I}_{N_Q} + k^2\mathsf{V}\mathsf{M}
\in \mathbb{C}^{N_Q \times N_Q},
\end{equation}
 the potential matrix $\mathsf{V} \in \mathbb{C}^{N_Q \times N_Q}$ has entries
$(\mathsf{V})_{jn} = \widetilde{\mathcal{V}}_\Omega[\mathbf{e}_n](y_j)$ (with $\mathbf{e}_n\in\mathbb{C}^{N_Q}$ the $n$-th canonical basis vector), and $(\mathbf{b})_j = u^{\inc}(y_j)$. Here and in the sequel, $\mathsf{I}_N$ denotes the identity in $\mathbb{C}^{N\times N}$. Once~\eqref{eq:nystrom_linear_system} has been solved, the total field at any point $x\in\R^2$ is recovered from the representation formula
\begin{equation}\label{eq:field_recovery}
  u(x) \approx u^{\inc}(x) - k^2\,\widetilde{\mathcal{V}}_\Omega[\mathsf{M}\mathbf{u}](x),
\end{equation}
a single additional application of the VDIM map.

\section{Preconditioner construction}\label{sec:preconditioner}

In this section we construct the preconditioner used to solve the linear system~\eqref{eq:nystrom_linear_system} by GMRES~\cite{saad1986gmres}. For concreteness and because this solver is publicly available, we build upon the sparsifying preconditioner of Ying~\cite{ying2015sparsifying}, but the method is agnostic to the underlying Cartesian discretization/solver. The idea is to embed the unstructured discretization of Section~\ref{sec:Nystrom_discretization} into a structured Cartesian framework, on which the fast preconditioner acts, through a pair of transfer operators between the quadrature nodes and a uniform Cartesian grid.

The structured grid is a uniform Cartesian grid $\mathcal{G}$ over $D$, with $N_C$ nodes $\{z_j\}_{j=1}^{N_C}$ and step size $h_C$. The step size is chosen so that $k\,\sup_{x\in\Omega}(n_1(x))^{1/2}\,h_C = \mathcal{O}(1)$, i.e., a fixed number of grid points per shortest local wavelength. The domains $\Omega$ and $D$ and their discretizations are illustrated in Figure~\ref{fig:interior_and_pixelation}. Together with the triangulation mesh size $h$ introduced above, the Cartesian step size
determines the \emph{mesh-size ratio}, defined as
\begin{equation}\label{eq:mesh_ratio}
  r_h := \frac{h}{h_C}.
\end{equation} 

We partition the grid nodes by their position relative to $\Omega$: $N_C^{\mathrm{int}}$ of them lie
strictly inside $\Omega$, and the remaining $N_C^{\mathrm{ext}}=N_C-N_C^{\mathrm{int}}$ lie in
$D\setminus\overline{\Omega}$. Ordering the grid so that the interior nodes come first, we write
\begin{equation}\label{eq:interior_cartesian_indices}
  \mathcal{I} := \{1,\ldots,N_C^{\mathrm{int}}\} \subset \{1,\ldots,N_C\}
\end{equation}
for the set of \emph{interior Cartesian indices}, so that $z_j$ lies strictly inside $\Omega$ for
$j\in\mathcal{I}$ and in $D\setminus\overline{\Omega}$ for $j\notin\mathcal{I}$. We write $z_j^{\mathrm{int}}:=z_j$ for $j\in\mathcal{I}$ to emphasize that the point is interior. Finally, the \emph{pixelation} of $\Omega$ is the set of Cartesian grid cells whose closure intersects the quadrature set $\mathcal{Q}$, and $\{z_j^{\mathrm{pix}}\}_{j=1}^{N_C^{\mathrm{pix}}}$
denotes their corner nodes; these subsets are
illustrated in Figure~\ref{fig:interior_and_pixelation}. Since a corner of a boundary cell may lie outside
$\overline{\Omega}$, we have $N_C^{\mathrm{int}} \leq N_C^{\mathrm{pix}}$.
\begin{figure}[ht]
    \centering
    \includegraphics[width=0.3\linewidth]{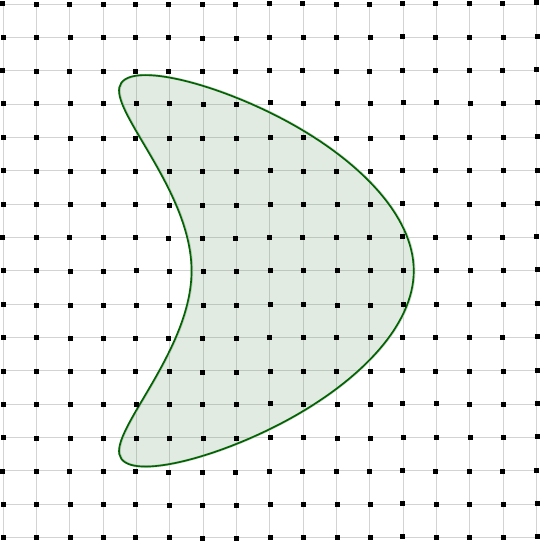}
    \hspace{0.07\linewidth}
    \includegraphics[width=0.3\linewidth]{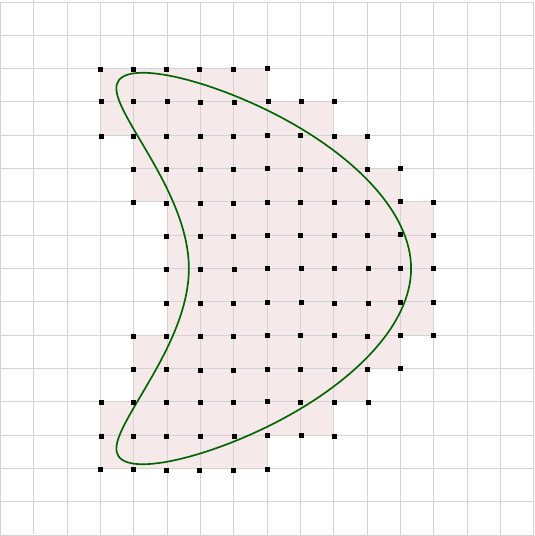}
    \caption{Cartesian grid points over the domain $D$ enclosing $\Omega$ (left), and their pixelation of $\Omega$ (right).}
    \label{fig:interior_and_pixelation}
\end{figure}

In this work we propose the preconditioner
\begin{equation}\label{eq:preconditioner_split}
    \mathsf{P} := \mathsf{T}_C^Q\mathsf{S}^{(C)}\mathsf{T}_Q^C
               + \bigl(\mathsf{I}_{N_Q} - \mathsf{T}_C^Q\mathsf{T}_Q^C\bigr),
\end{equation}
where $\mathsf{S}^{(C)} \in \mathbb{C}^{N_C \times N_C}$ is the sparsifying preconditioner~\cite{ying2015sparsifying} for the Nystr\"{o}m discretization on the uniform Cartesian grid~\cite{duan2009high}. The two rectangular operators transfer fields between the quadrature nodes and the Cartesian grid: $\mathsf{T}_C^Q \in \mathbb{R}^{N_Q \times N_C}$ maps a field from the Cartesian grid to the quadrature nodes, whereas $\mathsf{T}_Q^C \in \mathbb{R}^{N_C \times N_Q}$ maps it back. Here and in the sequel we follow the inter-grid transfer-operator convention standard in multigrid methods~\cite{briggs2000multigrid}, in which the subscript identifies the source point set and the superscript the target.

To expose the preconditioner's structure, note that the first term  in~\eqref{eq:preconditioner_split} is the natural embedding of the Cartesian solver---transfer from the quadrature nodes to the Cartesian grid, apply the sparsifying preconditioner, and transfer back---and the second term is a correction. If the transfer operators were exact mutual inverses, so that the round-trip $\mathsf{T}_C^Q\mathsf{T}_Q^C$ equalled $\mathsf{I}_{N_Q}$, this correction would vanish and $\mathsf{P}$ would reduce to the embedding $\mathsf{T}_C^Q\mathsf{S}^{(C)}\mathsf{T}_Q^C$. As we show below, however, such an identity cannot hold in our high-order setting, as the embedding term is itself singular.

\begin{proposition}\label{prop:non_invertible}
    If $N_Q > N_C^{\mathrm{pix}}$, then the embedding
    $\mathsf{T}_C^Q\mathsf{S}^{(C)}\mathsf{T}_Q^C\in\mathbb{C}^{N_Q\times N_Q}$ is not invertible.
\end{proposition}

\begin{proof}
    Only the $N_C^{\mathrm{pix}}$ pixelation nodes are vertices of a Cartesian cell containing a quadrature node, so every other column of $\mathsf{T}_C^Q$ vanishes and $\rank{\mathsf{T}_C^Q}\leq N_C^{\mathrm{pix}}$. Hence $\rank{\mathsf{T}_C^Q\mathsf{S}^{(C)}\mathsf{T}_Q^C}\leq N_C^{\mathrm{pix}} < N_Q$, and the operator is singular.
\end{proof}

Empirically, the choice of $\mathsf{T}_C^Q\mathsf{S}^{(C)}\mathsf{T}_Q^C$ does not lead to an effective preconditioner.

\begin{remark}
    The condition $N_Q > N_C^{\mathrm{pix}}$ holds in the typical configuration used in this
    work. Indeed, $N_Q = n_q N_T$ with $n_q=(p+1)(p+2)/2\geq6$ for interpolation degree $p\geq2$,
    while $N_C^{\mathrm{pix}}$ is comparable to the number of Cartesian cells meeting $\Omega$. In fact, 
    for $p\geq2$ each such cell contains several quadrature nodes, so $N_Q > N_C^{\mathrm{pix}}$.
\end{remark}
The correction term in~\eqref{eq:preconditioner_split} is consequently not a small perturbation but an essential ingredient that compensates for this rank deficiency.

We now proceed with a detailed description of the individual components of the proposed preconditioner.

\subsection{Sparsifying preconditioner}\label{subsec:sparsifying_preconditioner}

The Cartesian grid Nystr\"{o}m discretization of the
Lippmann--Schwinger equation takes the form $\mathsf{L}^{(C)}\mathbf{u}^{(C)} =
\mathbf{b}^{(C)}$, where
\begin{equation}\label{eq:LS_eq_on_cartesian_points}
  \mathsf{L}^{(C)} = \mathsf{I}_{N_C} + k^2\mathsf{V}^{(C)}\mathsf{M}^{(C)}
  \in \mathbb{C}^{N_C\times N_C},
\end{equation}
$\mathsf{V}^{(C)}$ is a Cartesian-grid discretization of the volume integral operator with the free-space Green's function~$G$, obtained via a high-order quadrature rule such as that of Duan--Rokhlin~\cite{duan2009high}, and $\mathsf{M}^{(C)}$
is the diagonal matrix with entries $\{m(z_j)\}_{j=1}^{N_C}$. In a nutshell, the sparsifying preconditioner
of~\cite{ying2015sparsifying} constructs a sparse matrix
$\mathsf{A} \in \mathbb{C}^{N_C\times N_C}$ having the same sparsity pattern as a
second-order compact finite-difference stencil for the Laplacian, with entries
chosen to approximately annihilate the off-diagonal blocks of $\mathsf{V}^{(C)}$.
Define the set of immediate Cartesian neighbors (stencils) of the vector index $\mathbf{j}\in\mathbb{N}^2$ as
\begin{equation}\label{eq:cart_neighbors_LS}
    \mathcal{N}(\mathbf{j}) = \bigl\{\mathbf{i}\in\mathbb{N}^2:
    \|\mathbf{j}-\mathbf{i}\|_{\infty}\leq 1,\;
    z_{\,\mathbf{i}}\in D\bigr\},
\end{equation}
where $z_{\mathbf{j}}$ denotes the Cartesian grid point with multi-index $\mathbf{j}=(j_1,j_2)$, $1\leq j_1\leq n_x$, $1\leq j_2\leq n_y$, and $n_x,\,n_y$ are the numbers of grid points in each direction, and let $\mathcal{N}(\mathbf{j})^c := \{\mathbf{i}: z_{\,\mathbf{i}}\in D\}\setminus\mathcal{N}(\mathbf{j})$ denote its complement, that is, the indices of all the Cartesian grid points lying outside the stencil of $\mathbf{j}$. The row $\mathsf{A}(\mathbf{j},\cdot)$ is then required to be supported in $\mathcal{N}(\mathbf{j})$ and to satisfy $\mathsf{A}(\mathbf{j},\mathcal{N}(\mathbf{j}))\,\mathsf{V}^{(C)}\bigl(\mathcal{N}(\mathbf{j}),\mathcal{N}(\mathbf{j})^c\bigr)\approx\mathbf{0}$; such a row exists because $G$ is smooth away from the origin, which renders the $|\mathcal{N}(\mathbf{j})|$ rows of that submatrix numerically rank deficient. Exploiting the translation invariance of $G$, it suffices to consider only the case $\mathbf{j}=\mathbf{0}$. In this case, the stencil weights are obtained from the problem
\begin{equation}\label{eq:sparsifying_min}
  \min_{\mathbf{v}\,:\,\|\mathbf{v}\|=1}
  \bigl\|\mathbf{v} \cdot \mathsf{V}^{(C)}\!\bigl(\mathcal{N}(\mathbf{0}),\mathcal{N}(\mathbf{0})^c\bigr)\bigr\|,
\end{equation}
which is solved via a singular value decomposition: $\mathbf v=\mathsf U(:,|\mathcal{N}(\mathbf{0})|)^*$, the left singular vector associated with the smallest singular value, where
$\mathsf{V}^{(C)}\!\bigl(\mathcal{N}(\mathbf{0}),\mathcal{N}(\mathbf{0})^c\bigr)=\mathsf U\Sigma \mathsf V^*$. Finally, $\mathsf{A}$ is built by setting $\mathsf{A}(\mathbf{j},\mathcal{N}(\mathbf{j}))=\mathbf{v}$ and the small entries of the product $\mathsf{A}\mathsf{L}^{(C)}$ can be filtered to produce the sparse matrix
$\mathsf{C}^{(C)} \approx \mathsf{A}\mathsf{L}^{(C)}$, with a sparsity pattern
resembling that of a finite-difference Helmholtz discretization.  The sparsifying
preconditioner is then defined as
\begin{equation}\label{eq:LC_def}
  \mathsf{S}^{(C)} := \bigl(\mathsf{C}^{(C)}\bigr)^{-1}\mathsf{A},
\end{equation}
so that $\mathsf{S}^{(C)}\mathsf{L}^{(C)} \approx \mathsf{I}_{N_C}$, i.e.,
$\mathsf{S}^{(C)}$ is an approximate inverse of $\mathsf{L}^{(C)}$ on the Cartesian
grid.  It was demonstrated in~\cite{ying2015sparsifying}, empirically, that the number of GMRES iterations
for the system preconditioned by $\mathsf{S}^{(C)}$ is nearly independent of the frequency.

Regarding cost, both $\mathsf{A}$ and $\mathsf{C}^{(C)}$ inherit the compact finite-difference sparsity pattern, so each application of $\mathsf{S}^{(C)}=(\mathsf{C}^{(C)})^{-1}\mathsf{A}$ amounts to a sparse matrix--vector product with $\mathsf{A}$, at $\mathcal{O}(N_C)$ cost, followed by the solution of the sparse system $\mathsf{C}^{(C)}$. In two dimensions the latter is carried out with a nested-dissection sparse direct solver, whose factorization is precomputed once at $\mathcal{O}(N_C^{3/2})$ cost and then reused at a cost of $\mathcal{O}(N_C\log N_C)$ per GMRES iteration. These factorization costs can be reduced to $\mathcal{O}(N_C)$ and the application costs to $\mathcal{O}(N_C)$ following~\cite{liu2018sparsify}.

\begin{remark}[Action of $\mathsf{S}^{(C)}$ outside the scatterer]\label{rmk:exterior_columns}
A structural feature of the sparsifying preconditioner, used repeatedly below, is that it
(approximately) ignores the values of its argument at the exterior Cartesian nodes. Indeed, since $\mathsf{C}^{(C)} \approx \mathsf{A} + k^2(\mathsf{A}\mathsf{V}^{(C)})\mathsf{M}^{(C)}$, we have
\begin{equation}\label{eq:S_minus_I}
  \mathsf{S}^{(C)} - \mathsf{I}_{N_C}
  = (\mathsf{C}^{(C)})^{-1}\bigl(\mathsf{A} - \mathsf{C}^{(C)}\bigr)
  \approx -k^2\,(\mathsf{C}^{(C)})^{-1}(\mathsf{A}\mathsf{V}^{(C)})\,\mathsf{M}^{(C)}.
\end{equation}
Therefore, since $\mathsf{M}^{(C)}=\operatorname{diag}(m(z_j))$ has zero columns at the exterior nodes, so does $\mathsf{S}^{(C)}-\mathsf{I}_{N_C}$. Equivalently, $\mathsf{S}^{(C)}\mathbf{v}$ depends on $\mathbf{v}$ only through its values at the Cartesian nodes inside the scatterer $\Omega$, and $\mathsf{S}^{(C)}$ acts as the identity on any $\mathbf{v}$ vanishing at those nodes. (The solution operator of the continuous Lippmann--Schwinger equation has the same property, by the resolvent identity~\eqref{eq:S_resolvent_identity}.)
Throughout the paper we assume that this structure holds exactly, that is, that $\mathsf{S}^{(C)}-\mathsf{I}_{N_C}$ has identically zero columns at the exterior Cartesian
nodes; in practice we see values of a small factor times machine epsilon.

This property matters for the preconditioner. Collecting the two terms
of~\eqref{eq:preconditioner_split} gives the equivalent factored form
\begin{equation}\label{eq:preconditioner_factored}
    \mathsf{P} = \mathsf{I}_{N_Q} + \mathsf{T}_C^Q(\mathsf{S}^{(C)} - \mathsf{I}_{N_C})\mathsf{T}_Q^C,
\end{equation}
which is the form used throughout the analysis below. Since the quadrature nodes all lie in $\Omega$ while the Cartesian grid covers $D$, the operator $\mathsf{T}_Q^C$ must produce values at Cartesian nodes far from the scatterer---values that its input, supported entirely in $\Omega$, leaves undetermined. By~\eqref{eq:S_minus_I}, the rows of $\mathsf{T}_Q^C$ at those nodes are annihilated in the product $(\mathsf{S}^{(C)}-\mathsf{I}_{N_C})\mathsf{T}_Q^C$, through which $\mathsf{T}_Q^C$ enters $\mathsf{P}$. The extension outside the scatterer therefore has no effect on $\mathsf{P}$, and each construction of Section~\ref{subsec:T_Q^C} may fix it by whatever convention is most convenient.
\end{remark}

\subsection{Cartesian-to-quadrature transfer operator}\label{subsec:T_C^Q}

The operator $\mathsf{T}_C^Q \in \mathbb{R}^{N_Q \times N_C}$ maps field values at the
Cartesian points $\{z_j\}_{j=1}^{N_C}$ to the quadrature nodes $\{y_q\}_{q=1}^{N_Q}$ by bilinear interpolation.
For each Cartesian grid point $z_j=(z_{j,1},z_{j,2})$, $j\in\{1,\ldots,N_C\}$, let
$\psi_j:\mathbb{R}^2\to\mathbb{R}$ denote the associated $\mathcal{Q}_1$ (bilinear) Lagrange
nodal basis function. On each Cartesian cell $\sigma\in\mathcal{G}$ having $z_j$ as a
vertex,
\begin{equation}\label{eq:psi_def}
    \psi_j(x_1,x_2)
    := \left(1-\frac{\alpha}{h_C}\right)\left(1-\frac{\beta}{h_C}\right),
    \qquad \alpha:=|x_1-z_{j,1}|,\quad \beta:=|x_2-z_{j,2}|,
\end{equation}
for $(x_1,x_2)\in\sigma$. Outside the union of such cells $\psi_j$ is extended by zero, and it
satisfies $\psi_j(z_k)=\delta_{jk}$ for all $j,k\in\{1,\ldots,N_C\}$.
Given a quadrature node $y_q\in\Omega$, let $\sigma$ be the Cartesian cell containing it. The four vertices of $\sigma$ are grid points, so row $q$ of $\mathsf{T}_C^Q$ is supported on their four indices, with entries the bilinear weights $(\mathsf{T}_C^Q)_{qj}=\psi_j(y_q)$ (all four are nonzero unless $y_q$ lies on a grid line, in which case the vertices off that line receive zero weight).

Since every row of $\mathsf{T}_C^Q$ sums to one, the operator reproduces constant
fields, $\mathsf{T}_C^Q(c\mathbf{1}_{N_C}) = c\mathbf{1}_{N_Q}$ for any
$c\in\mathbb{C}$, where $\mathbf{1}_N\in\mathbb{R}^N$ denotes the vector of all ones. 
Clearly, $\mathsf{T}_C^Q$ is sparse---each row has at most four nonzero entries---and its only nonzero columns are those of the $N_C^{\mathrm{pix}}$ pixelation nodes, so it holds $\mathcal{O}(N_Q)$ nonzeros and each application costs $\mathcal{O}(N_Q)$ operations.

\subsection{Quadrature-to-Cartesian transfer operator}\label{subsec:T_Q^C}

The operator $\mathsf{T}_Q^C \in \mathbb{R}^{N_C \times N_Q}$ maps field values at the quadrature nodes back to the Cartesian grid. Its rows indexed by $\mathcal{I}$ are required to preserve constant fields, i.e.,
\begin{equation}\label{eq:roundtrip_constants_int}
  \bigl(\mathsf{T}_Q^C(\mathbf{1}_{N_Q})\bigr)_j = 1
  \quad\text{for every }j\in\mathcal{I},
\end{equation}
which ensures that the round trip
\begin{equation}\label{eq:round_trip_def}
  \mathsf{R}^{(C)} := \mathsf{T}_Q^C\mathsf{T}_C^Q \in \mathbb{R}^{N_C\times N_C}
\end{equation}
also preserves constants on $\mathcal{I}$.

Condition~\eqref{eq:roundtrip_constants_int} is the minimal accuracy constraint one can impose on a transfer operator. Reproduction of constants is the zeroth-order case of polynomial reproduction, and it is the standard requirement for inter-grid transfers in multigrid~\cite{briggs2000multigrid}.

In what follows, we present two constructions of $\mathsf{T}_Q^C$. Other choices are certainly possible, but the two considered here are arguably the most natural. They employ different field-mapping mechanisms, distinguished by the superscripts $\mathrm{I}$ and $\mathrm{II}$ in the notation of the resulting operators.

\subsubsection{Approach~I: Local projection}\label{ssubsec:local_proj}
This construction maps the density values at the quadrature nodes to the interior
Cartesian nodes by first projecting, element by element and in the $L^2$ sense, onto
piecewise-linear functions, and then evaluating the projection at the Cartesian
nodes. To express this operation in matrix form, let $\{\phi_i^{(\tau)}\}_{i=1}^3$ denote the
local $\mathcal{P}_1$ Lagrange basis on the triangle $\tau\in\mathcal{T}$, defined by
$\phi_i^{(\tau)}(v_j^{(\tau)})=\delta_{ij}$ at the vertices $\{v_j^{(\tau)}\}_{j=1}^3$
of $\tau$ and extended by zero outside $\overline{\tau}$ (a discontinuous Galerkin
basis), and index the corresponding $3N_T$ degrees of freedom by the pairs
$(\tau,i)$, $\tau\in\mathcal{T}$, $i\in\{1,2,3\}$, where $N_T = |\mathcal{T}|$.
The interior block
$\widetilde{\mathsf{T}}_Q^{C,\mathrm{I}} \in \mathbb{R}^{N_C^{\mathrm{int}} \times N_Q}$
is then defined by the composition
\begin{equation}\label{eq:T_HC_proj_def}
  \widetilde{\mathsf{T}}_Q^{C,\mathrm{I}} :=
  \mathsf{E}_C\,(\Mtri)^{-1}\mathsf{E}_Q^{\top}\,\mathsf{W},
\end{equation}
where:
\begin{itemize}
  \item $\mathsf{W} = \operatorname{diag}(\omega_q)_{q=1}^{N_Q} \in
        \mathbb{R}^{N_Q \times N_Q}$ is the diagonal matrix of quadrature weights;
  \item $\mathsf{E}_Q \in \mathbb{R}^{N_Q \times 3N_T}$ is the evaluation matrix of the
        basis $\{\phi_i^{(\tau)}\}_{i=1}^3$ at the quadrature nodes $y_q$, with entries
        $(\mathsf{E}_Q)_{q,(\tau,i)} = \phi_i^{(\tau)}(y_q)$ if $y_q\in\tau$ and zero
        otherwise;
  \item $\Mtri = \operatorname{diag}\bigl(\mathsf{M}^{(\tau)}\bigr)_{\tau\in\mathcal{T}}
        \in \mathbb{R}^{3N_T \times 3N_T}$ is the block-diagonal mass matrix, whose local
        blocks satisfy  $(\mathsf{M}^{(\tau)})_{ij} \approx
        \int_\tau \phi_i^{(\tau)}\phi_j^{(\tau)}\,\de y$; and
  \item $\mathsf{E}_C \in \mathbb{R}^{N_C^{\mathrm{int}} \times 3N_T}$ is the evaluation
        matrix of the basis $\{\phi_i^{(\tau)}\}_{i=1}^3$ at the interior Cartesian nodes $z_j^{\mathrm{int}}$, with entries
        $(\mathsf{E}_C)_{j,(\tau,i)} = \phi_i^{(\tau)}(z_j^{\mathrm{int}})$ if
        $z_j^{\mathrm{int}}\in\tau$ and zero otherwise, with an arbitrary but fixed choice of element when $z_j^{\mathrm{int}}$ lies on an inter-element edge.
\end{itemize}

The approximation sign in the definition of $\Mtri$ accounts for the fact that the
local blocks are evaluated treating each element as a straight-edged
triangle. In that case
\begin{equation}\label{eq:mass_closed_form}
  \bigl(\mathsf{M}^{(\tau)}\bigr)_{ij}
  = \begin{cases}
      \dfrac{|\tau|}{6},  & i = j,\\[2mm]
      \dfrac{|\tau|}{12}, & i \neq j,
    \end{cases}
  \qquad\text{so that}\qquad
  \bigl(\mathsf{M}^{(\tau)}\bigr)^{-1}
  = \frac{3}{|\tau|}
    \begin{bmatrix} 3 & -1 & -1\\ -1 & 3 & -1\\ -1 & -1 & 3 \end{bmatrix}.
\end{equation}

The composition~\eqref{eq:T_HC_proj_def} acts as follows: for the vector
$\mathbf{f}\in\mathbb{C}^{N_Q}$ of density values $(\mathbf{f})_q=f(y_q)$ (cf.\ Section~\ref{sec:Nystrom_discretization}), the entries
$(\mathsf{E}_Q^{\top}\mathsf{W}\mathbf{f})_{(\tau,i)} =
\sum_{q\,:\,y_q\in\tau}\omega_q\,\phi_i^{(\tau)}(y_q)\,(\mathbf{f})_q$ form the
quadrature approximation of the load vector $\int_\tau \phi_i^{(\tau)} f\,\de y$,
so that $(\Mtri)^{-1}\mathsf{E}_Q^{\top}\mathsf{W}\mathbf{f}$ collects the
coefficients of the (quadrature-based) elementwise $L^2$-projection of $f$ onto
piecewise-linear functions, which $\mathsf{E}_C$ finally evaluates at the interior
Cartesian nodes.

The full operator has the block structure
\begin{equation}\label{eq:T_HC_I_block}
  \mathsf{T}_Q^{C,\mathrm{I}} :=
  \begin{bmatrix}
    \widetilde{\mathsf{T}}_Q^{C,\mathrm{I}} \\
    \mathsf{0}
  \end{bmatrix} \in \mathbb{R}^{N_C\times N_Q},
\end{equation}
with the rows indexed by $\mathcal{I}$ given by $\widetilde{\mathsf{T}}_Q^{C,\mathrm{I}}$ and an arbitrary---here, zero---choice for the remaining $N_C^{\mathrm{ext}}$ exterior-node rows.
This choice has no effect on $\mathsf{P}$, by Remark~\ref{rmk:exterior_columns}. Indeed, the zero columns of $\mathsf{S}^{(C)}-\mathsf{I}_{N_C}$ at the exterior Cartesian nodes annihilate the corresponding rows of any matrix multiplying it from the right, so the exterior rows of $\mathsf{T}_Q^{C,\mathrm{I}}$, and likewise those of the round trip $\mathsf{R}^{(C)}_{\mathrm{I}}$ defined below, never reach the preconditioner~\eqref{eq:preconditioner_factored}. 

Moreover, since $\mathsf M^{\rm tri}$ is block-diagonal with respect to the
elements $\tau\in\mathcal T$, and since the entries
$(\mathsf E_C)_{j,(\tau,i)}$ and $(\mathsf E_Q)_{q,(\tau,i)}$ vanish whenever
$z_j\notin\tau$ and $y_q\notin\tau$, respectively, only those elements $\tau$
containing both $z_j$ and $y_q$ contribute to the $(j,q)$ entry.
Consequently,
\begin{equation}\label{eq:sparse_I}
  \bigl(\widetilde{\mathsf T}_Q^{C,\mathrm I}\bigr)_{jq}=0
  \qquad\text{whenever}\qquad
  \{z_j,y_q\}\not\subset\tau \quad\text{for all }\tau\in\mathcal T.
\end{equation}

In what follows we denote the round-trip operator associated with this construction by
\begin{equation}\label{eq:T^(C)_I}
  \mathsf{R}^{(C)}_{\mathrm{I}} := \mathsf{T}_Q^{C,\mathrm{I}}\mathsf{T}_C^Q.
\end{equation}

\subsubsection{Approach~II: Renormalized transpose}\label{ssubsec:renormalized_transpose}
We define $\mathsf{T}_Q^{C,\mathrm{II}} \in \mathbb{R}^{N_C \times N_Q}$ as the
renormalized transpose of $\mathsf{T}_C^Q$:
\begin{equation}\label{eq:T_H^C_def}
  \mathsf{T}_Q^{C,\mathrm{II}} := \mathsf{D}^{-1}(\mathsf{T}_C^Q)^{\top},
\end{equation}
where $\mathsf{D}\in\mathbb{R}^{N_C\times N_C}$ is the diagonal matrix with entries
\begin{equation}\label{eq:D_def}
  (\mathsf{D})_{jj} =
  \begin{cases}
    \displaystyle\sum_{q=1}^{N_Q}(\mathsf{T}_C^Q)_{qj}, & \text{$j$ a pixelation node},\\[4pt]
    1, & \text{otherwise}.
  \end{cases}
\end{equation}
The column sum is positive exactly at the pixelation nodes and vanishes at the remaining Cartesian nodes, which no quadrature node reaches. At those remaining nodes we make the arbitrary choice of $1$, which keeps $\mathsf{D}$ invertible. Since the columns of $\mathsf{T}_C^Q$, hence the rows of $(\mathsf{T}_C^Q)^{\top}$, vanish at those remaining nodes, the value assigned there does not affect $\mathsf{T}_Q^{C,\mathrm{II}}$.

As before, we denote the round-trip operator associated with this construction by
\begin{equation}\label{eq:T^(C)_II}
  \mathsf{R}^{(C)}_{\mathrm{II}} := \mathsf{T}_Q^{C,\mathrm{II}}\mathsf{T}_C^Q.
\end{equation}

\begin{remark}[Choice of the transfer operator]\label{rmk:choice_of_mu}
After a comprehensive comparison, we use Approach~II in all numerical experiments of
Section~\ref{sec:numerical_examples}. As documented in
Section~\ref{subsec:operator_selection}, both constructions yield an invertible preconditioner, but Approach~II gives a better conditioned system $\mathsf{P}\mathsf{L}$, with condition number nearly constant under mesh refinement, and consistently fewer GMRES iterations, with counts that remain nearly unchanged as the frequency and the material
  contrast increase. The two also differ in the hypotheses under which we are able to establish invertibility, none of which appear to be necessary in practice. For Approach~II, invertibility is guaranteed at every mesh ratio, provided the form of $\mathsf{S}^{(C)}$ in a suitable inner product avoids the non-positive real axis (Theorem~\ref{thm:P_invertible_unconditional}), whereas for Approach~I it is guaranteed for a sufficiently small mesh ratio $r_h$~\eqref{eq:mesh_ratio}, with no assumption on $\mathsf{S}^{(C)}$ beyond its invertibility (Theorem~\ref{thm:approachI_invertibility}).
\end{remark}
\section{Invertibility of the preconditioner}\label{sec:invertibility}

In this section we establish sufficient conditions for the invertibility of the preconditioner defined in~\eqref{eq:preconditioner_factored}. We begin by reducing the question to the invertibility of a matrix on the Cartesian grid, of size $N_C\times N_C$.

\begin{lemma}\label{lem:well_poss_P}
Let $\mathsf{S}^{(C)}$, $\mathsf{T}_C^Q$, and $\mathsf{T}_Q^C$ be as defined in
Section~\ref{sec:preconditioner}, let $\mathsf{P}$ be the
preconditioner~\eqref{eq:preconditioner_split}, and let
$\mathsf{R}^{(C)}=\mathsf{T}_Q^C\mathsf{T}_C^Q$ be the round trip~\eqref{eq:round_trip_def}.
Define
\begin{equation}\label{eq:def_Q}
    \mathsf{Q} := \mathsf{I}_{N_C} + (\mathsf{S}^{(C)} - \mathsf{I}_{N_C})\mathsf{R}^{(C)}
    \in\mathbb{C}^{N_C\times N_C}.
\end{equation}
Then $\mathsf{P}$ is invertible if and only if $\mathsf{Q}$ is.
\end{lemma}
\begin{proof}
Set $\mathsf{A} := \mathsf{T}_C^Q\in\mathbb{R}^{N_Q\times N_C}$ and
$\mathsf{B} := (\mathsf{S}^{(C)}-\mathsf{I}_{N_C})\mathsf{T}_Q^C\in\mathbb{C}^{N_C\times N_Q}$.
Then $\mathsf{P}=\mathsf{I}_{N_Q}+\mathsf{A}\mathsf{B}$, while, since
$\mathsf{T}_Q^C\mathsf{T}_C^Q=\mathsf{R}^{(C)}$,
\[
    \mathsf{B}\mathsf{A}
    = (\mathsf{S}^{(C)}-\mathsf{I}_{N_C})\mathsf{T}_Q^C\mathsf{T}_C^Q
    = (\mathsf{S}^{(C)}-\mathsf{I}_{N_C})\mathsf{R}^{(C)},
\]
so $\mathsf{Q}=\mathsf{I}_{N_C}+\mathsf{B}\mathsf{A}$. By the Weinstein--Aronszajn (Sylvester)
determinant identity~\cite{horn2012matrix}, $\det(\mathsf{I}_{N_Q}+\mathsf{A}\mathsf{B})
=\det(\mathsf{I}_{N_C}+\mathsf{B}\mathsf{A})$, that is $\det\mathsf{P}=\det\mathsf{Q}$. The
two determinants therefore vanish together, so $\mathsf{P}$ is nonsingular if and only if
$\mathsf{Q}$ is.
\end{proof}

By Lemma~\ref{lem:well_poss_P}, the invertibility of $\mathsf{P}$ reduces to that of $\mathsf{Q}$. We derive sufficient conditions for the latter separately for each of the two constructions of $\mathsf{T}_Q^C$ introduced in Section~\ref{sec:preconditioner}, presented in the same order as their definitions: Section~\ref{subsec:analysis_approach_I} treats Approach~I and Section~\ref{subsec:analysis_approach_II} treats Approach~II.

\subsection{Analysis of Approach~I}\label{subsec:analysis_approach_I}

The analysis of Approach~I relies on a perturbation argument. Definition~\eqref{eq:def_Q}
can be rewritten as
$\mathsf{Q}=\mathsf{S}^{(C)}-(\mathsf{S}^{(C)}-\mathsf{I}_{N_C})(\mathsf{I}_{N_C}-\mathsf{R}^{(C)})$,
which exhibits $\mathsf{Q}$ as a perturbation of the invertible matrix $\mathsf{S}^{(C)}$,
small whenever the round trip is close to the identity. The next lemma makes this precise.
Only the rows of $\mathsf{I}_{N_C}-\mathsf{R}^{(C)}$ indexed by the interior Cartesian index set $\mathcal{I}$ in~\eqref{eq:interior_cartesian_indices} enter the
resulting condition, since $\mathsf{S}^{(C)}-\mathsf{I}_{N_C}$ annihilates the remaining ones.

\begin{lemma}\label{lem:Q_invertibility}
Let $\mathsf{S}^{(C)}\in\mathbb{C}^{N_C\times N_C}$ be invertible, let
$\mathsf{R}^{(C)}_{\mathrm{I}}\in\mathbb{R}^{N_C\times N_C}$ be the round-trip
matrix~\eqref{eq:T^(C)_I}, and let $\mathsf{Q}$ be as in~\eqref{eq:def_Q} with
$\mathsf{R}^{(C)}=\mathsf{R}^{(C)}_{\mathrm{I}}$. Set
$\mathsf{Z}:=\mathsf{I}_{N_C}-\mathsf{R}^{(C)}_{\mathrm{I}}$ and let
$\mathsf{Z}_{\mathcal{I}}\in\mathbb{R}^{N_C^{\mathrm{int}}\times N_C}$ be the submatrix
of $\mathsf{Z}$ formed by the rows indexed
by~$\mathcal{I}$~\eqref{eq:interior_cartesian_indices}. Then $\mathsf{Q}$ is
invertible whenever
\begin{equation}\label{eq:invertibility_condition_unsplit}
\bigl\|(\mathsf{S}^{(C)})^{-1}(\mathsf{S}^{(C)}-\mathsf{I}_{N_C})\mathsf{Z}\bigr\|_\infty < 1
\end{equation}
and, in particular, whenever the more restrictive condition
\begin{equation}\label{eq:invertibility_condition_T}
\bigl\|(\mathsf{S}^{(C)})^{-1}(\mathsf{S}^{(C)}-\mathsf{I}_{N_C})\bigr\|_\infty\,\bigl\|\mathsf{Z}_{\mathcal{I}}\bigr\|_\infty
< 1
\end{equation}
holds.
\end{lemma}
\begin{proof}
Writing $\mathsf{R}^{(C)}_{\mathrm{I}}=\mathsf{I}_{N_C}-\mathsf{Z}$ in~\eqref{eq:def_Q} and using the
invertibility of $\mathsf{S}^{(C)}$,
\[
    \mathsf{Q}
    = \mathsf{I}_{N_C}+(\mathsf{S}^{(C)}-\mathsf{I}_{N_C})(\mathsf{I}_{N_C}-\mathsf{Z})
    = \mathsf{S}^{(C)}\bigl[\mathsf{I}_{N_C} - (\mathsf{S}^{(C)})^{-1}(\mathsf{S}^{(C)}-\mathsf{I}_{N_C})\mathsf{Z}\bigr],
\]
so, by the Neumann series, $\mathsf{Q}$ is invertible
under~\eqref{eq:invertibility_condition_unsplit}. That~\eqref{eq:invertibility_condition_T}
implies~\eqref{eq:invertibility_condition_unsplit} follows by splitting the product
submultiplicatively, as we now show.
Let $\mathsf{\Pi}$ denote the diagonal projector onto the interior indices, i.e.,
$\mathsf{\Pi}_{jj}=1$ for $j\in\mathcal{I}$ and $\mathsf{\Pi}_{jj}=0$ otherwise. Since
$\mathsf{S}^{(C)}-\mathsf{I}_{N_C}$ has zero columns off $\mathcal{I}$ by~\eqref{eq:S_minus_I}, it
satisfies $(\mathsf{S}^{(C)}-\mathsf{I}_{N_C})\mathsf{\Pi}=\mathsf{S}^{(C)}-\mathsf{I}_{N_C}$, whereas
$\mathsf{\Pi}\mathsf{Z}$ is $\mathsf{Z}$ with its exterior rows set to zero, so that
$\|\mathsf{\Pi}\mathsf{Z}\|_\infty=\|\mathsf{Z}_{\mathcal{I}}\|_\infty$. Inserting $\mathsf{\Pi}$ and
using the submultiplicativity of the induced $\infty$-norm therefore gives
\[
    \bigl\|(\mathsf{S}^{(C)})^{-1}(\mathsf{S}^{(C)}-\mathsf{I}_{N_C})\mathsf{Z}\bigr\|_\infty
    = \bigl\|(\mathsf{S}^{(C)})^{-1}(\mathsf{S}^{(C)}-\mathsf{I}_{N_C})\,\mathsf{\Pi}\mathsf{Z}\bigr\|_\infty
    \leq \bigl\|(\mathsf{S}^{(C)})^{-1}(\mathsf{S}^{(C)}-\mathsf{I}_{N_C})\bigr\|_\infty\,\bigl\|\mathsf{Z}_{\mathcal{I}}\bigr\|_\infty,
\]
whose right-hand side is smaller than one exactly under~\eqref{eq:invertibility_condition_T}.
\end{proof}

It remains to bound $\|\mathsf{Z}_{\mathcal{I}}\|_\infty$ for the Approach~I transfer operator
$\mathsf{T}_Q^{C,\mathrm{I}}$ of~\eqref{eq:T_HC_I_block}. We begin with a row-sum property of
$\mathsf{T}_Q^{C,\mathrm{I}}$, on which that bound is based. We note that all its hypotheses on the quadrature
rule are met by the VR rule of Section~\ref{sec:Nystrom_discretization}.

\begin{lemma}
  \label{lem:rowsum_THC}
Assume that all elements $\tau\in\mathcal{T}$ are straight-edged, so that the local mass
matrices are given exactly by~\eqref{eq:mass_closed_form}, and that the quadrature rule has
positive weights $\omega_q>0$, has all of its nodes in the interiors of the elements, and
satisfies $d_{\mathrm{ex}}(p)\geq1$, so that it integrates linear polynomials exactly on each
element. Then for any $\tau\in\mathcal{T}$,
\begin{equation}\label{eq:quad_linear_exact}
    \sum_{q\,:\,y_q\in\tau}\phi_i^{(\tau)}(y_q)\,\omega_q
    = \int_\tau \phi_i^{(\tau)}(y)\de y
    = \frac{|\tau|}{3},
    \qquad i\in\{1,2,3\},
\end{equation}
where $\{\phi_i^{(\tau)}\}_{i=1}^3$ are the $\mathcal{P}_1$ (linear Lagrange) shape functions associated with the triangle $\tau$.
Moreover,
\begin{equation}\label{eq:rowsum_THC}
    \sum_{q=1}^{N_Q}\bigl(\mathsf{T}_Q^{C,\mathrm{I}}\bigr)_{jq} = 1,
    \qquad j\in\mathcal{I}.
\end{equation}
\end{lemma}
\begin{proof}
Let $j\in\mathcal{I}$ and let $\tau\subset\overline\Omega$ be the element assigned to $z_j$
in the definition of $\mathsf{E}_C$~\eqref{eq:T_HC_proj_def}, so that
$(\mathsf{E}_C)_{j,(\tau',i)}=0$ for every $\tau'\neq\tau$. Since $\Mtri$ is block diagonal,
only the block of $\tau$ contributes to row $j$, and we can write the entries of block
$\widetilde{\mathsf{T}}_Q^{C,\mathrm{I}}$ in~\eqref{eq:T_HC_I_block} as
\begin{equation}\label{eq:expression_T_HC_proj}\begin{split}
    \bigl(\widetilde{\mathsf{T}}_Q^{C,\mathrm{I}}\bigr)_{jq}
    &= \frac{3\,\omega_q}{|\tau|}\sum_{i=1}^{3}\phi_i^{(\tau)}(z_j)
      \Bigl(3\phi_i^{(\tau)}-\phi_{i'}^{(\tau)}-\phi_{i''}^{(\tau)}\Bigr)(y_q)\quad \left(\{i,i',i''\}=\{1,2,3\}\right)\\
      &=\frac{3\,\omega_q}{|\tau|}\sum_{i=1}^{3}\phi_i^{(\tau)}(z_j)
      \Bigl(4\phi_i^{(\tau)}(y_q)-1\Bigr),
\end{split}
\end{equation}
which follows from~\eqref{eq:T_HC_proj_def},~\eqref{eq:mass_closed_form}, and the fact that $\sum_i\phi_i^{(\tau)}(y_q)=1$. (Here the sum runs over $i\in\{1,2,3\}$, and $i',i''$ denote the two remaining indices; by the identity $\sum_i\phi_i^{(\tau)}(y_q)=1$, each term inside the parentheses equals $4\phi_i^{(\tau)}(y_q)-1$.)

Only the quadrature nodes lying in $\tau$ contribute, since
$(\mathsf{E}_Q)_{q,(\tau,i)}=0$ otherwise, and each $y_q$ lies in exactly one element
because the quadrature nodes are interior to the elements. Restricting the sum
in~\eqref{eq:rowsum_THC} accordingly, we get
\[
    \sum_{q=1}^{N_Q}\bigl(\mathsf{T}_Q^{C,\mathrm{I}}\bigr)_{jq}
    = \sum_{i=1}^{3}\phi_i^{(\tau)}(z_j)\,\frac{3}{|\tau|}
      \sum_{y_q\in\tau}\Bigl(4\phi_i^{(\tau)}(y_q)-1\Bigr)\,\omega_q.
\]
Since $\tau$ is straight-edged, condition~\eqref{eq:quad_linear_exact} applies and the inner sum above evaluates to
\[
    \sum_{y_q\in\tau}\Bigl(4\phi_i^{(\tau)}(y_q)-1\Bigr)\,\omega_q
    = \frac{|\tau|}{3}.
\]
Substituting back and using again the identity $\sum_{i=1}^{3}\phi_i^{(\tau)}(z_j)=1$, we arrive at~\eqref{eq:rowsum_THC}. 
\end{proof}

Building on the unit row sums of Lemma~\ref{lem:rowsum_THC}, we can now show that, at
interior Cartesian nodes, the round-trip operator
$\mathsf{R}^{(C)}_{\mathrm{I}}$ approaches the identity as $r_h\to 0$.

\begin{lemma}\label{lem:approachI_estimate}
Let the assumptions of Lemma~\ref{lem:rowsum_THC} hold. Then, for every $h<h_C$, the
round-trip operator $\mathsf{R}^{(C)}_{\mathrm{I}}$ defined in~\eqref{eq:T^(C)_I} satisfies
\begin{equation}\label{eq:IminusTc_exact}
    \bigl\|\mathsf{Z}_{\mathcal{I}}\bigr\|_\infty
    = \max_{j\in\mathcal{I}}\biggl(
    \bigl|1-\bigl(\mathsf{R}^{(C)}_{\mathrm{I}}\bigr)_{jj}\bigr|
    +\sum_{j':j'\neq j}
     \bigl|\bigl(\mathsf{R}^{(C)}_{\mathrm{I}}\bigr)_{jj'}\bigr|\biggr)
    \leq 14\sqrt2\, r_h,
\end{equation}
where $\mathsf{Z}$ and $\mathsf{Z}_{\mathcal{I}}$ are as in
Lemma~\ref{lem:Q_invertibility} and $r_h = h/h_C$ is the mesh
ratio~\eqref{eq:mesh_ratio}. In particular, $\mathsf{R}^{(C)}_{\mathrm{I}}$ approaches the
identity at every interior node as $r_h\to0$.
\end{lemma}

\begin{proof}
  We consider indices $j\in\mathcal{I}$ throughout, since
  these are the only ones entering the norm $\|\mathsf{Z}_{\mathcal{I}}\|_\infty$.
  To bound it, we distinguish between diagonal and non-diagonal terms in the row-sum of $\mathsf{I}_{N_C}-\mathsf{R}^{(C)}_{\mathrm{I}}$. For diagonal terms, using the unit row sum of $\mathsf{T}_Q^{C,\mathrm{I}}$ from Lemma~\ref{lem:rowsum_THC}, it follows that
\begin{equation}\label{eq:diag_identity}
    1 - \bigl(\mathsf{R}^{(C)}_{\mathrm{I}}\bigr)_{jj}
    = \sum_{q=1}^{N_Q}\bigl(\mathsf{T}_Q^{C,\mathrm{I}}\bigr)_{jq}
      - \sum_{q=1}^{N_Q}\bigl(\mathsf{T}_Q^{C,\mathrm{I}}\bigr)_{jq}
        \bigl(\mathsf{T}_C^Q\bigr)_{qj}
    = \sum_{q=1}^{N_Q}\bigl(\mathsf{T}_Q^{C,\mathrm{I}}\bigr)_{jq}
      \Bigl(1-\bigl(\mathsf{T}_C^Q\bigr)_{qj}\Bigr).
\end{equation}
Since $(\mathsf{T}_C^Q)_{qj}\in[0,1]$, the triangle inequality applied to~\eqref{eq:diag_identity} yields
\begin{equation}\label{eq:diag_bound_step3}
    \bigl|1-\bigl(\mathsf{R}^{(C)}_{\mathrm{I}}\bigr)_{jj}\bigr|
    \leq \sum_{q=1}^{N_Q}
         \bigl|\bigl(\mathsf{T}_Q^{C,\mathrm{I}}\bigr)_{jq}\bigr|
         \Bigl(1-\bigl(\mathsf{T}_C^Q\bigr)_{qj}\Bigr).
\end{equation}

For the off-diagonal terms, on the other hand, we have
\begin{align}\label{eq:offdiag_bound_step3}
    \sum_{j':j'\neq j}\bigl|\bigl(\mathsf{R}^{(C)}_{\mathrm{I}}\bigr)_{jj'}\bigr|
    &= \sum_{j':j'\neq j}
       \Biggl|\sum_{q=1}^{N_Q}\bigl(\mathsf{T}_Q^{C,\mathrm{I}}\bigr)_{jq}
              \bigl(\mathsf{T}_C^Q\bigr)_{qj'}\Biggr|
    \leq \sum_{j':j'\neq j}\sum_{q=1}^{N_Q}
         \bigl|\bigl(\mathsf{T}_Q^{C,\mathrm{I}}\bigr)_{jq}\bigr|
         \bigl(\mathsf{T}_C^Q\bigr)_{qj'} \notag\\
    &= \sum_{q=1}^{N_Q}
       \bigl|\bigl(\mathsf{T}_Q^{C,\mathrm{I}}\bigr)_{jq}\bigr|
       \underbrace{\sum_{j':j'\neq j}
       \bigl(\mathsf{T}_C^Q\bigr)_{qj'}}_{
       =\,1-(\mathsf{T}_C^Q)_{qj}}
    = \sum_{q=1}^{N_Q}
      \bigl|\bigl(\mathsf{T}_Q^{C,\mathrm{I}}\bigr)_{jq}\bigr|
      \Bigl(1-\bigl(\mathsf{T}_C^Q\bigr)_{qj}\Bigr).
\end{align}

Adding \eqref{eq:diag_bound_step3} and \eqref{eq:offdiag_bound_step3}, we arrive at
\begin{align}\label{eq:norm_bound_split}
    \bigl|1-\bigl(\mathsf{R}^{(C)}_{\mathrm{I}}\bigr)_{jj}\bigr|
    +\sum_{j':j'\neq j}
     \bigl|\bigl(\mathsf{R}^{(C)}_{\mathrm{I}}\bigr)_{jj'}\bigr|
    &\leq 2
       \sum_{q=1}^{N_Q}
       \bigl|(\mathsf{T}_Q^{C,\mathrm{I}})_{jq}\bigr|\,\Bigl(1-\bigl(\mathsf{T}_C^Q\bigr)_{qj}\Bigr).
\end{align}

Then, recall from~\eqref{eq:sparse_I} that  $(\mathsf{T}_Q^{C,\mathrm{I}})_{jq}\neq0$ only if
$y_q$ lies in the element $\tau$ assigned to $z_j$, as in the proof of
Lemma~\ref{lem:rowsum_THC}. Therefore,
\begin{equation}\label{eq:l1_restriction}
    \sum_{q=1}^{N_Q}
    \bigl|\bigl(\mathsf{T}_Q^{C,\mathrm{I}}\bigr)_{jq}\bigr|
         \Bigl(1-\bigl(\mathsf{T}_C^Q\bigr)_{qj}\Bigr)
    = \sum_{q\,:\,y_q\in\tau}
      \bigl|\bigl(\mathsf{T}_Q^{C,\mathrm{I}}\bigr)_{jq}\bigr|
         \Bigl(1-\bigl(\mathsf{T}_C^Q\bigr)_{qj}\Bigr).
\end{equation}
To bound the term in parentheses above we note that, since $h < h_C$, for every $q\in\{1,\ldots,N_Q\}$ with
$y_q\in\tau\ni z_j$, it holds that
$\|y_q - z_j\|_\infty \leq h < h_C$,
so $y_q$ lies in the Cartesian cell
$\sigma\in\mathcal{G}$ having $z_j$ as a corner,
and formula \eqref{eq:psi_def} applies with $\alpha_q := |z_{j,1} - y_{q,1}|$ and
$\beta_q := |z_{j,2} - y_{q,2}|$, both in $[0,h]$. Expressing the bilinear weight
$\bigl(\mathsf{T}_C^Q\bigr)_{qj}=\psi_j(y_q)$ in terms of $\alpha_q$ and $\beta_q$, we get
\begin{equation}\label{eq:bilinear_exact}
    0\leq 1-\bigl(\mathsf{T}_C^Q\bigr)_{qj}
    = \frac{\alpha_q+\beta_q}{h_C}      
      - \frac{\alpha_q\beta_q}{h_C^2}
    \leq \frac{\|z_j-y_q\|_1}{h_C}\leq \sqrt2\frac{\|z_j-y_q\|_2}{h_C}\leq\sqrt2\frac{h}{h_C}.      
\end{equation}
The last inequality holds because both $z_j$ and $y_q$ lie in the same triangle $\tau$, so their Euclidean distance is bounded by the diameter of $\tau$, which in turn is bounded by the mesh size $h$. 

Using this bound in~\eqref{eq:l1_restriction} we arrive at
\begin{equation}\label{eq:sum_bound_sqrt2}
\sum_{q=1}^{N_Q}
    \bigl|\bigl(\mathsf{T}_Q^{C,\mathrm{I}}\bigr)_{jq}\bigr|
         \Bigl(1-\bigl(\mathsf{T}_C^Q\bigr)_{qj}\Bigr)\leq     \sqrt2\, r_h \sum_{q\,:\,y_q\in\tau}
      \bigl|\bigl(\mathsf{T}_Q^{C,\mathrm{I}}\bigr)_{jq}\bigr|.
\end{equation}

Applying the triangle inequality to the entrywise
expression~\eqref{eq:expression_T_HC_proj}, using $\omega_q>0$ together with
$\phi_i^{(\tau)}(z_j)\geq0$ and $|4\phi_i^{(\tau)}-1|\leq 4\phi_i^{(\tau)}+1$ on $\tau$, then
using~\eqref{eq:quad_linear_exact} (valid since
$d_{\mathrm{ex}}(p)\geq1$ and $\tau$ is straight-edged) and the
identity $\sum_{i=1}^3\phi_i^{(\tau)}(z_j)=1$,
\begin{equation}\label{eq:l1_bound_five}
    \sum_{q\,:\,y_q\in\tau}
    \bigl|\bigl(\mathsf{T}_Q^{C,\mathrm{I}}\bigr)_{jq}\bigr|
    \leq \sum_{i=1}^{3}\phi_i^{(\tau)}(z_j)\,\frac{3}{|\tau|}
      \sum_{q\,:\,y_q\in\tau}
      \Bigl(4\phi_i^{(\tau)}(y_q)+1\Bigr)
      \,\omega_q
    = 7 .
\end{equation}

Combining \eqref{eq:sum_bound_sqrt2} and~\eqref{eq:l1_bound_five}, and substituting
into~\eqref{eq:norm_bound_split}, whose right-hand side is now independent of
$j$, and taking the maximum over $j\in\mathcal{I}$, we finally arrive at
\begin{equation*}
    \bigl\|\mathsf{Z}_{\mathcal{I}}\bigr\|_\infty
    \leq 2\cdot7\cdot\sqrt{2}\,r_h
    = 14\sqrt{2}\,r_h.
\end{equation*}
The proof is complete.
\end{proof}

The following result establishes the invertibility of $\mathsf{Q}$ and hence of the preconditioner $\mathsf{P}$
for Approach~I.

\begin{theorem}
  \label{thm:approachI_invertibility}
Fix $h_C>0$, assume $\mathsf{S}^{(C)}$ is
invertible, and let the assumptions of Lemmas~\ref{lem:rowsum_THC}
and~\ref{lem:approachI_estimate} hold. Then there exists $h_0>0$ such that, for
every $h<h_0$, the invertibility
condition~\eqref{eq:invertibility_condition_T} of Lemma~\ref{lem:Q_invertibility} is
satisfied. Consequently
$\mathsf{Q}$ and the preconditioner
$\mathsf{P}$~\eqref{eq:preconditioner_split} are invertible.
\end{theorem}
\begin{proof}
By Lemma~\ref{lem:approachI_estimate}, $\|\mathsf{Z}_{\mathcal{I}}\|_\infty\leq
14\sqrt2\,r_h\to0$ as $r_h\to 0$, so, with $h_C$ fixed, this quantity tends to zero as
$h\to0$. Since $\mathsf{S}^{(C)}$ is fixed, the product
$\|(\mathsf{S}^{(C)})^{-1}(\mathsf{S}^{(C)}-\mathsf{I}_{N_C})\|_\infty$ appearing in
Lemma~\ref{lem:Q_invertibility} is likewise a fixed positive constant.
Therefore, for $h$ small enough,
$\|(\mathsf{S}^{(C)})^{-1}(\mathsf{S}^{(C)}-\mathsf{I}_{N_C})\|_\infty\|\mathsf{Z}_{\mathcal{I}}\|_\infty<1$
and condition~\eqref{eq:invertibility_condition_T} holds. This
yields invertibility of $\mathsf{Q}$ and hence, by Lemma~\ref{lem:well_poss_P}, of the preconditioner~$\mathsf{P}$.
\end{proof}

\begin{remark}[Curved vs.\ polygonal domains]\label{rmk:approachI_curved}
The estimate of Lemma~\ref{lem:approachI_estimate} assumes the setting of a straight-edged tessellation, hence a polygonal domain. More generally,
on a curved element, let $F:\hat\tau\to\tau$ be the curvilinear map from the fixed reference triangle $\hat\tau$, and let $A:\hat\tau\to\tau$ be the affine map that agrees with $F$ at the three vertices, whose image is the straight triangle used in the polygonal analysis. For a shape-regular curvilinear mesh resolving a boundary of uniformly bounded curvature, standard element-map estimates (e.g.~\cite[Ch.~4]{Sauter2010}) give $\|F-A\|_{L^\infty(\hat\tau)}=\mathcal O(h^2)$, so the quadrature nodes $F(\hat y_q)$ are displaced by $\mathcal  O(h^2)$ from their straight-triangle positions $A(\hat y_q)$, while the element area is perturbed to $|\tau|\to|\tau|\,(1+\mathcal O(h))$. Since the $\mathcal{P}_1$ basis functions are Lipschitz with constant $\mathcal  O(h^{-1})$, both effects enter the entries of $\mathsf{T}_Q^{C,\mathrm{I}}$ as an $\mathcal O(h)$ relative perturbation, and the bound~\eqref{eq:IminusTc_exact} gets an additional term,
\begin{equation}\label{eq:IminusTc_curved}
    \bigl\|\mathsf{Z}_{\mathcal{I}}\bigr\|_\infty
    \leq 14\sqrt2\,r_h + \mathcal O(h),
\end{equation}
which vanishes at least linearly as $h\to0$ and hence leaves the first-order rate unchanged. Since the right-hand side of~\eqref{eq:IminusTc_curved} still tends to zero as $h\to0$,  it can be used, in place of~\eqref{eq:IminusTc_exact}, to achieve the smallness condition on $h$ in Theorem~\ref{thm:approachI_invertibility}, thereby establishing the invertibility of $\mathsf{Q}$ and of the preconditioner $\mathsf{P}$ on curved meshes for all $h$ small enough. Numerical evidence of this linear decay is presented in Appendix~\ref{app:precond_validation}, where Table~\ref{tab:invertibility_certificate} reports $\|\mathsf{Z}_{\mathcal{I}}\|_\infty$ for the unit disk---a curved domain, discretized with curved elements---decreasing linearly in the mesh ratio $r_h$ (with a fixed $h_C$), in agreement with~\eqref{eq:IminusTc_curved}.
\end{remark}

\subsection{Analysis of Approach~II}\label{subsec:analysis_approach_II}

The analysis of Approach~II is structurally different. In place of the perturbation argument above, it exploits the Gram structure
$\mathsf{D}\,\mathsf{R}^{(C)}_{\mathrm{II}}=(\mathsf{T}_C^Q)^\top\mathsf{T}_C^Q$, which renders $\mathsf{R}^{(C)}_{\mathrm{II}}$ self-adjoint and positive semidefinite in a suitable inner product for every mesh size. Together with the condition~\eqref{eq:diss_on_range} on $\mathsf{S}^{(C)}$, which we assume, motivated by the energy balance obeyed by the continuous solution operator that $\mathsf{S}^{(C)}$ approximates (Section~\ref{sec:dissipativity_physical}), this yields the invertibility of $\mathsf{Q}$ in~\eqref{eq:def_Q}, and hence of the preconditioner~$\mathsf{P}$.

\begin{lemma}
  \label{lem:approachII_gram}
For the renormalized transpose $\mathsf{T}_Q^{C,\mathrm{II}}=\mathsf{D}^{-1}(\mathsf{T}_C^Q)^\top$ in~\eqref{eq:T_H^C_def}, consider the Hermitian inner product on $\mathbb{C}^{N_C}$ induced by the positive diagonal matrix $\mathsf{D}$ in~\eqref{eq:D_def},
\begin{equation}\label{eq:D_inner_product}
  \langle\mathbf{x},\mathbf{y}\rangle_{\mathsf{D}}=\mathbf{x}^{*}\mathsf{D}\,\mathbf{y},
  \qquad \mathbf{x},\mathbf{y}\in\mathbb{C}^{N_C}.
\end{equation}
The round-trip operator $\mathsf{R}^{(C)}_{\mathrm{II}}=\mathsf{T}_Q^{C,\mathrm{II}}\mathsf{T}_C^Q$ in~\eqref{eq:round_trip_def} is self-adjoint and positive semidefinite with respect to $\langle\cdot,\cdot\rangle_{\mathsf{D}}$, and it is a contraction:
\begin{equation}\label{eq:roundtrip_contraction}
  \bigl\|\mathsf{R}^{(C)}_{\mathrm{II}}\bigr\|_{\mathsf{D}}\leq1 ,
\end{equation}
where $\|\cdot\|_{\mathsf{D}}$ denotes the operator norm induced by
$\langle\cdot,\cdot\rangle_{\mathsf{D}}$.
\end{lemma}
\begin{proof}
Since $\mathsf{T}_Q^{C,\mathrm{II}}=\mathsf{D}^{-1}(\mathsf{T}_C^Q)^{\top}$, we have $\mathsf{D}\,\mathsf{R}^{(C)}_{\mathrm{II}} = (\mathsf{T}_C^Q)^{\top}\mathsf{T}_C^Q$, which is real symmetric. Hence, for all $\mathbf{x},\mathbf{y}\in\mathbb{C}^{N_C}$,
$\langle\mathsf{R}^{(C)}_{\mathrm{II}}\mathbf{x},\mathbf{y}\rangle_{\mathsf{D}}
  =\mathbf{x}^{*}(\mathsf{T}_C^Q)^{\top}\mathsf{T}_C^Q\,\mathbf{y} =\langle\mathbf{x},\mathsf{R}^{(C)}_{\mathrm{II}}\mathbf{y}\rangle_{\mathsf{D}},
$
so $\mathsf{R}^{(C)}_{\mathrm{II}}$ is self-adjoint with respect to $\langle\cdot,\cdot\rangle_{\mathsf{D}}$. Moreover, for every $\mathbf{x}\in\mathbb{C}^{N_C}$,
$\langle\mathsf{R}^{(C)}_{\mathrm{II}}\mathbf{x},\mathbf{x}\rangle_{\mathsf{D}}   =\mathbf{x}^{*}(\mathsf{T}_C^Q)^{\top}\mathsf{T}_C^Q\mathbf{x}
  =\bigl\|\mathsf{T}_C^Q\mathbf{x}\bigr\|_2^{2}\geq0,$ so $\mathsf{R}^{(C)}_{\mathrm{II}}$ is positive semidefinite with respect to $\langle\cdot,\cdot\rangle_{\mathsf{D}}$. 
  
On the other hand, in view of the positivity of the diagonal matrix $\mathsf D$, the eigenvalues $\lambda_i$, $i\in\{1,\ldots,N_C\}$, of $\mathsf R^{(C)}_{\rm II}$ are real and non-negative. To bound them from above, note that all entries of
$\mathsf{R}^{(C)}_{\mathrm{II}}=\mathsf{D}^{-1}(\mathsf{T}_C^Q)^{\top}\mathsf{T}_C^Q$ are
non-negative, and that its row sums are
\[
  \bigl(\mathsf{R}^{(C)}_{\mathrm{II}}\mathbf{1}_{N_C}\bigr)_j
  =\bigl(\mathsf{D}^{-1}(\mathsf{T}_C^Q)^{\top}\mathbf{1}_{N_Q}\bigr)_j
  =\frac{1}{(\mathsf{D})_{jj}}\sum_{q=1}^{N_Q}(\mathsf{T}_C^Q)_{qj}
  =\begin{cases}
     1, & \text{$j$ a pixelation node},\\
     0, & \text{otherwise},
   \end{cases}
\]
where the first equality uses the unit row sums of $\mathsf{T}_C^Q$ and the last one the
definition~\eqref{eq:D_def} of $\mathsf{D}$.
Therefore
$\|\mathsf{R}^{(C)}_{\mathrm{II}}\|_{\infty}=\|\mathsf{R}^{(C)}_{\mathrm{II}}\mathbf{1}_{N_C}\|_{\infty}=1$, and since the spectral radius is bounded by any
induced norm, $0\leq\lambda_i\leq1$ for all $i\in\{1,\ldots,N_C\}$. Finally, since $\mathsf{R}^{(C)}_{\rm II}$ is positive semidefinite and self-adjoint with respect to the $\mathsf{D}$-inner product $\langle\cdot,\cdot\rangle_{\mathsf{D}}$, its operator norm induced by $\langle\cdot,\cdot\rangle_{\mathsf{D}}$ equals its largest eigenvalue, so
$\bigl\|\mathsf{R}^{(C)}_{\mathrm{II}}\bigr\|_{\mathsf{D}}=\max_{i\in\{1,\ldots,N_C\}}\lambda_i\leq1$, which is~\eqref{eq:roundtrip_contraction}.
\end{proof}

It remains to combine the Gram structure and the contraction bound of Lemma~\ref{lem:approachII_gram} with a condition on the sparsifying preconditioner $\mathsf{S}^{(C)}$ excluding a specific behavior of its $\mathsf{D}$-form. Together they yield the main invertibility result of this section, which, unlike Theorem~\ref{thm:approachI_invertibility}, carries no smallness requirement on $h$.

\begin{theorem}
  \label{thm:P_invertible_unconditional}
Fix $h_C>0$. Let $\langle\cdot,\cdot\rangle_{\mathsf{D}}$ be the inner product of Lemma~\ref{lem:approachII_gram}, with respect to which the round trip $\mathsf{R}^{(C)}_{\mathrm{II}}$ is self-adjoint, positive semidefinite and a contraction. If the form of the sparsifying preconditioner $\mathsf{S}^{(C)}$ with respect to $\langle\cdot,\cdot\rangle_{\mathsf{D}}$ avoids the non-positive real axis on the range of the round trip, i.e.\
\begin{equation}\label{eq:diss_on_range}
  \langle\mathsf{S}^{(C)}\mathbf{w},\mathbf{w}\rangle_{\mathsf{D}}\notin(-\infty,0]
  \quad\text{for all nonzero }\mathbf{w}\in\operatorname{range}\mathsf{R}^{(C)}_{\mathrm{II}},
\end{equation}
then, for \emph{every} triangulation mesh size $h>0$, the preconditioner $\mathsf{P}$ in~\eqref{eq:preconditioner_split}, associated to Approach~II, is invertible.
\end{theorem}
\begin{proof}
By Lemma~\ref{lem:well_poss_P} it suffices to show that $\mathsf{Q}$ in~\eqref{eq:def_Q} is invertible, with $\mathsf{R}^{(C)}=\mathsf{R}^{(C)}_{\mathrm{II}}$. Writing $\mathsf{Q}=\mathsf{I}_{N_C}-(\mathsf{I}_{N_C}-\mathsf{S}^{(C)})\mathsf{R}^{(C)}_{\mathrm{II}}$, this amounts to $1\notin\sigma\bigl((\mathsf{I}_{N_C}-\mathsf{S}^{(C)})\mathsf{R}^{(C)}_{\mathrm{II}}\bigr)$. As $\mathsf{R}^{(C)}_{\mathrm{II}}$ is self-adjoint and positive semidefinite with respect to $\langle\cdot,\cdot\rangle_{\mathsf{D}}$ (Lemma~\ref{lem:approachII_gram}), it admits a square root $\mathsf{F}=(\mathsf{R}^{(C)}_{\mathrm{II}})^{1/2}$ that is itself self-adjoint and positive semidefinite with respect to $\langle\cdot,\cdot\rangle_{\mathsf{D}}$, with $\mathsf{R}^{(C)}_{\mathrm{II}}=\mathsf{F}^{2}$. Using the spectrum identity $\sigma(\mathsf{AB})=\sigma(\mathsf{BA})$, which holds for any two square matrices $\mathsf{A}$ and $\mathsf{B}$, we have
\begin{equation}
    \sigma\bigl((\mathsf{I}_{N_C}-\mathsf{S}^{(C)})\mathsf{F}^{2}\bigr)
    = \sigma\bigl(\mathsf{F}(\mathsf{I}_{N_C}-\mathsf{S}^{(C)})\mathsf{F}\bigr),
\end{equation}
so it suffices to show that $1\notin\sigma(\mathsf{M})$, where $\mathsf{M}:=\mathsf{F}(\mathsf{I}_{N_C}-\mathsf{S}^{(C)})\mathsf{F}$. Suppose, for contradiction, that $\mathsf{M}\mathbf{v}=\mathbf{v}$ for some $\mathbf{v}\neq\mathbf{0}$, and set $\mathbf{w}:=\mathsf{F}\mathbf{v}\in\operatorname{range}\mathsf{F}=\operatorname{range}\mathsf{R}^{(C)}_{\mathrm{II}}$. Since $\mathsf{F}$ is $\langle\cdot,\cdot\rangle_{\mathsf{D}}$-self-adjoint and $\mathsf{M}\mathbf{v}=\mathbf{v}$,
\begin{equation}\label{eq:Qdiss_key}
    \langle\mathbf{v},\mathbf{v}\rangle_{\mathsf{D}}
    =\langle\mathsf{M}\mathbf{v},\mathbf{v}\rangle_{\mathsf{D}}
    =\bigl\langle(\mathsf{I}_{N_C}-\mathsf{S}^{(C)})\mathbf{w},\mathbf{w}\bigr\rangle_{\mathsf{D}}
    =\langle\mathbf{w},\mathbf{w}\rangle_{\mathsf{D}}-\langle\mathsf{S}^{(C)}\mathbf{w},\mathbf{w}\rangle_{\mathsf{D}}.
\end{equation}
Because $\mathsf{F}$ may be singular, we distinguish two cases. If $\mathbf{w}=\mathsf{F}\mathbf{v}=\mathbf{0}$, then $\mathbf{v}$ lies in the kernel of $\mathsf{F}$ and the right-hand side of~\eqref{eq:Qdiss_key} vanishes. Taking real parts then forces $\langle\mathbf{v},\mathbf{v}\rangle_{\mathsf{D}}=0$, hence $\mathbf{v}=\mathbf{0}$, a contradiction. If instead $\mathbf{w}\neq\mathbf{0}$, we read off separately the imaginary and the real part of~\eqref{eq:Qdiss_key}. Since $\langle\mathbf{v},\mathbf{v}\rangle_{\mathsf{D}}$ and $\langle\mathbf{w},\mathbf{w}\rangle_{\mathsf{D}}$ are both real, the imaginary part gives $\imag\langle\mathsf{S}^{(C)}\mathbf{w},\mathbf{w}\rangle_{\mathsf{D}}=0$, so that $\langle\mathsf{S}^{(C)}\mathbf{w},\mathbf{w}\rangle_{\mathsf{D}}$ is real, while the real part gives
\begin{equation}\label{eq:Qdiss_real}
    \langle\mathsf{S}^{(C)}\mathbf{w},\mathbf{w}\rangle_{\mathsf{D}}
    =\langle\mathbf{w},\mathbf{w}\rangle_{\mathsf{D}}-\langle\mathbf{v},\mathbf{v}\rangle_{\mathsf{D}} .
\end{equation}
The right-hand side of~\eqref{eq:Qdiss_real} is non-positive. Indeed, $\mathsf{F}$ is $\langle\cdot,\cdot\rangle_{\mathsf{D}}$-self-adjoint with $\mathsf{F}^{2}=\mathsf{R}^{(C)}_{\mathrm{II}}$, so $\mathbf{w}=\mathsf{F}\mathbf{v}$ satisfies $\langle\mathbf{w},\mathbf{w}\rangle_{\mathsf{D}}=\langle\mathsf{R}^{(C)}_{\mathrm{II}}\mathbf{v},\mathbf{v}\rangle_{\mathsf{D}}$, and the Cauchy--Schwarz inequality together with the contraction bound~\eqref{eq:roundtrip_contraction} of Lemma~\ref{lem:approachII_gram} yields
    $\langle\mathbf{w},\mathbf{w}\rangle_{\mathsf{D}} =\langle\mathsf{R}^{(C)}_{\mathrm{II}}\mathbf{v},\mathbf{v}\rangle_{\mathsf{D}}\leq\bigl\|\mathsf{R}^{(C)}_{\mathrm{II}}\bigr\|_{\mathsf{D}}\,\langle\mathbf{v},\mathbf{v}\rangle_{\mathsf{D}}\leq\langle\mathbf{v},\mathbf{v}\rangle_{\mathsf{D}}$.
Hence $\langle\mathsf{S}^{(C)}\mathbf{w},\mathbf{w}\rangle_{\mathsf{D}}$ is real and non-positive, that is, it lies in $(-\infty,0]$, which contradicts~\eqref{eq:diss_on_range} since $\mathbf{w}\in\operatorname{range}\mathsf{R}^{(C)}_{\mathrm{II}}$ and $\mathbf{w}\neq\mathbf{0}$. In either case we reach a contradiction, so $1\notin\sigma(\mathsf{M})$, and thus $\mathsf{Q}$, and therefore $\mathsf{P}$, is invertible.
\end{proof}

\subsubsection{Energy balance behind the hypothesis of Theorem~\ref{thm:P_invertible_unconditional}}\label{sec:dissipativity_physical}
This section connects the hypothesis~\eqref{eq:diss_on_range} of Theorem~\ref{thm:P_invertible_unconditional} to the continuous scattering problem~\eqref{eq:bvp}, thereby identifying conditions on the contrast $m$ under which~\eqref{eq:diss_on_range} is to be expected. The link is made through the imaginary-part condition
\begin{equation}\label{eq:diss_imag_part}
  \imag\langle\mathsf{S}^{(C)}\mathbf{w},\mathbf{w}\rangle_{\mathsf{D}}\neq0
  \quad\text{for all nonzero }\mathbf{w}\in\operatorname{range}\mathsf{R}^{(C)}_{\mathrm{II}},
\end{equation}
which is more restrictive than, and hence sufficient for, the condition~\eqref{eq:diss_on_range} actually assumed in Theorem~\ref{thm:P_invertible_unconditional}.

Expressing the Lippmann--Schwinger equation~\eqref{eq:Lippmann-Schwinger_small} in operator form as
\begin{equation}\label{eq:LS_operator_form}
  (\mathcal{I}+k^2\mathcal{V}_\Omega\mathcal{M})\,u = u^{\inc}\ \ \text{in }D,
  \qquad \mathcal{M}g:=m\,g ,
\end{equation}
with $\mathcal{V}_\Omega$ the volume potential~\eqref{eq:vol_pot}, we let $\mathcal{S}:=(\mathcal{I}+k^2\mathcal{V}_\Omega\mathcal{M})^{-1}:L^2(D)\to L^2(D)$ be the solution operator, mapping an incident field $u^{\inc}$ to the total field $u=\mathcal{S}u^{\inc}$. It is $\mathcal{S}$ that $\mathsf{S}^{(C)}$ approximates, being by~\eqref{eq:LC_def} an approximate inverse of the Cartesian discretization $\mathsf{L}^{(C)}$~\eqref{eq:LS_eq_on_cartesian_points} of~\eqref{eq:LS_operator_form} on $D$. 
Moreover, the scattered field $u^s=u-u^{\inc}\in H^2_{\mathrm{loc}}(\R^2)$ satisfies $(\Delta+k^2n)\,u^s=k^2m\,u^{\inc}$ in $\R^2$ together with the radiation condition~\eqref{eq:sommerfeld_rc_def}, so $\mathcal{S}$ admits the form
\begin{equation}\label{eq:S_resolvent_identity}
  \mathcal{S}=\mathcal{I}+k^2\mathcal{R}\mathcal{M}
  \quad\text{on } L^2(D),
\end{equation}
where $\mathcal{R}$ denotes the outgoing resolvent of the full medium, taking a compactly supported $f\in L^2(\R^2)$ to the unique radiating solution $v\in H^2_{\mathrm{loc}}(\R^2)$ of $(\Delta+k^2n)v=f$ in $\R^2$~\cite{COLTON:2012,cakoni2016inverse}. Since $\mathcal{M}$ is the rightmost factor in~\eqref{eq:S_resolvent_identity}, and it annihilates any function vanishing on $\overline\Omega$, the difference $\mathcal{S}-\mathcal{I}$ ignores the values of its argument outside the scatterer. This is the counterpart of the vanishing exterior columns of $\mathsf{S}^{(C)}-\mathsf{I}_{N_C}$ in~\eqref{eq:S_minus_I}. 

For $u^{\inc}\in L^2(D)$, write $u:=\mathcal{S}u^{\inc}$ and let $u_\infty\in L^2(\mathbb{S}^1)$ be the far-field pattern of the scattered field $u^s=u-u^{\inc}$~\cite{COLTON:2012}. Then Green's first identity on a large ball, together with the radiation condition, yields the energy balance:
\begin{equation}\label{eq:dissipativity_identity}
  \imag\int_D\overline{m\,\mathcal{S}u^{\inc}}\;u^{\inc}\,\de x
  \;=\;\imag\int_\Omega\overline{m\,\mathcal{S}u^{\inc}}\;u^{\inc}\,\de x
  \;=\;\int_\Omega(\imag n)\,|u|^2\,\de x+\frac1k\,\|u_\infty\|^2_{L^2(\mathbb{S}^1)}
  \;\geq\;0.
\end{equation} Both terms on the right are non-negative, the first by the standing assumption $\imag n=n_2/k\geq0$ of Section~\ref{sec:problem_setup}, which reflects the $\e^{-i\omega t}$ time convention. They vanish simultaneously only if the medium fails to absorb ($(\imag n)\,|u|^2\equiv0$ in $\Omega$, as holds in particular for a non-absorbing medium $\imag n\equiv0$) and $u^{\inc}$ is a non-scattering incident field ($u_\infty\equiv0$). For incident fields solving the Helmholtz equation this degeneracy is precisely the interior transmission eigenvalue problem, and it can occur only at the (discrete) set of transmission eigenvalues~\cite{cakoni2016inverse}.

Two gaps separate the energy balance~\eqref{eq:dissipativity_identity} from hypothesis~\eqref{eq:diss_on_range}. These are (a) the contrast weight $m$ in its leftmost integral, absent from $\langle\mathsf{S}^{(C)}\mathbf{w},\mathbf{w}\rangle_{\mathsf{D}}$, and (b) the passage from an integral over $D$ to the discrete $\mathsf{D}$-weighted pairing.

For (a), split the contrast as $m=m_0+\delta m$ with $m_0$ a real constant, so that~\eqref{eq:dissipativity_identity} becomes
\begin{equation}\label{eq:variable_contrast_perturbation}
  \imag\int_\Omega\overline{\mathcal{S}u^{\inc}}\;u^{\inc}\,\de x
  =\frac{1}{m_0}\Bigl(\int_\Omega(\imag n)|u|^2\de x+\frac1k\|u_\infty\|^2
    -\imag\int_\Omega\overline{\delta m\,u}\;u^{\inc}\,\de x\Bigr),
\end{equation}
so the unweighted pairing retains a definite sign, that of $m_0$, as long as the contrast variation $\delta m$ is small relative to the absorbed and radiated energy.

For (b), recall from~\eqref{eq:D_def} that $(\mathsf{D})_{jj}=\sum_q\psi_j(y_q)\approx\frac{N_Q}{|\Omega|}\int_\Omega\psi_j\de x$ at the pixelation nodes, so $\langle\cdot,\cdot\rangle_{\mathsf{D}}$ is a lumped-mass discretization of the $L^2(\Omega)$ inner product of the $\mathcal{Q}_1$ interpolants $x_h:=\sum_j x_j\psi_j$. Moreover, every $\mathbf{w}\in\operatorname{range}\mathsf{R}^{(C)}_{\mathrm{II}}$ vanishes at the non-pixelation nodes, so $w_h$ is supported in the union of the pixelation cells, a region containing $\Omega$ whose boundary lies at distance $\mathcal{O}(h_C)$ from $\partial\Omega$, and the pairing over $D$ reduces to one over $\Omega$ up to that $\mathcal{O}(h_C)$ discrepancy. Taking $u^{\inc}:=w_h$, noting that the quadrature points are nearly uniformly distributed, and interpreting $\mathsf{S}^{(C)}\mathbf{w}$ as the nodal values of $\mathcal{S}u^{\inc}$, we get
\begin{equation}\label{eq:discrete_to_continuous_pairing}
  \langle\mathsf{S}^{(C)}\mathbf{w},\mathbf{w}\rangle_{\mathsf{D}}
  \;\approx\;\frac{N_Q}{|\Omega|}\int_\Omega\overline{\mathcal{S}u^{\inc}}\;u^{\inc}\,\de x ,
\end{equation}
which is precisely the quantity in~\eqref{eq:variable_contrast_perturbation}. The energy balance thus supplies the mechanism behind~\eqref{eq:diss_on_range} as well as the sign of the imaginary part, but it does not certify the hypothesis, since the passage to the discrete pairing carries discretization error and $\mathsf{S}^{(C)}$ is a sparsified approximate inverse rather than a discretized solution operator. 

Appendix~\ref{app:precond_validation} probes~\eqref{eq:diss_on_range} directly, by sampling the normalized form $\langle\mathsf{S}^{(C)}\mathbf{w},\mathbf{w}\rangle_{\mathsf{D}}/\|\mathbf{w}\|^2_{\mathsf{D}}$ over $\operatorname{range}\mathsf{R}^{(C)}_{\mathrm{II}}$. It is observed that where the contrast has a definite sign, the sampled imaginary parts carry that sign, as predicted by~\eqref{eq:variable_contrast_perturbation}, and stay away from zero, so that~\eqref{eq:diss_imag_part} holds. In turn, where the contrast changes sign, the imaginary part does too, and it is instead a positive real part that keeps the form off the non-positive real axis. While the numerical evidence supports the hypothesis~\eqref{eq:diss_on_range} in the cases considered, it is not claimed to hold in general.

\section{Numerical examples}
\label{sec:numerical_examples}

In this section, we present a variety of numerical examples designed to validate and test  the capabilities of the proposed preconditioner for the high-order Nystr\"om discretization of the Lippmann--Schwinger equation~\eqref{eq:nystrom_linear_system}. The first part examines its performance on a benchmark  problem, plane-wave scattering by a constant-contrast disk, for which a closed-form  Fourier--Bessel series solution is available.  The second part presents examples that demonstrate that the preconditioning strategy extends to more general and challenging problems.

Unless otherwise specified, all examples in this section use the following
setup. The high-order VDIM potential map~\eqref{eq:DIM_map} is applied
matrix-free through the FMM (FMM2D), accessed via v0.2.0 of the
\texttt{Inti.jl} package~\cite{Inti} using Julia 1.12.5. We use a VR quadrature
order supporting VDIM interpolation order $p=2$ (six nodes per element) and a
Cartesian grid resolution set by $kh_C=0.125$. All boundary term corrections
in~\eqref{eq:DIM_decomp} are performed via the general-purpose density
interpolation method~\cite{faria2021general} with a boundary correction radius
of $7h$ for the layer potentials entering~\eqref{eq:nystrom_linear_system}
via~\eqref{eq:DIM_decomp} and of $14h$ for the field
recovery~\eqref{eq:representation_formula}. The latter is a single
post-processing apply, carried out only where an exterior error or a field plot
is reported and excluded from the reported solve times.  The linear
system~\eqref{eq:nystrom_linear_system} is solved by left-preconditioned GMRES
without restart. All computations are performed on a MacBook Pro with 16
physical cores and 48 GB of RAM, with the exception of the simulation depicted
in Figure~\ref{fig:yinyang} which was performed on a machine with 64GB of
memory.

\subsection{Unit disk benchmarks}\label{subsec:benchmark}

To compare the solution produced by our preconditioned solver against a reliable reference, we take $\Omega$ to be the unit disk with constant refractive index $n=n_1=1+\eta$ inside and $n=1$ outside, so that the contrast satisfies $\eta:=n_1-1=-m>0$, and let the incident field be the plane wave $u^{\inc}(x)=\e^{ik\,x\cdot d}$ with direction of propagation $d=(1,0)$. The domain is exactly discretized with the curved (isogeometric) triangular elements supported in \texttt{Inti.jl}. Since the contrast is constant, the interior wavenumber is $k_{\mathrm{int}}:=k\sqrt{1+\eta}=k\sqrt{n_1}$ and the reference solution $u_{\mathrm{ref}}$ is available as a Fourier--Bessel series which can be evaluated to any desired accuracy.

With a reference solution at hand, we measure the accuracy of our preconditioned VDIM solver by
the relative maximum-norm error of its solution $u$ with respect to $u_{\mathrm{ref}}$, taken over
a finite set $\mathcal{X}$ of \emph{target points},
\begin{equation}\label{eq:error_metrics}
  e^{\mathrm{rel}}(\mathcal{X})
    :=\frac{\max_{x\in\mathcal{X}}\,\bigl|u(x)-u_{\mathrm{ref}}(x)\bigr|}
           {\max_{x\in\mathcal{X}}\,\bigl|u_{\mathrm{ref}}(x)\bigr|}.
\end{equation}
The set $\mathcal{X}$ is defined explicitly wherever an error is reported, since it is chosen
according to what the experiment is meant to demonstrate. 

Two choices of $\mathcal{X}$ recur throughout: the \emph{interior} set, the quadrature nodes
themselves, and the \emph{exterior} set, the Cartesian nodes lying outside the scatterer,
\begin{equation}\label{eq:measurement_sets}
  \mathcal{X}_{\mathrm{int}}:=\mathcal{Q}=\{y_q\}_{q=1}^{N_Q}\subset\Omega,
  \quad
  \mathcal{X}_{\mathrm{ext}}:=\{\,z_j\ :\ j\notin\mathcal{I}\,\}\subset D\setminus\overline{\Omega},
\end{equation}
with the corresponding errors~\eqref{eq:error_metrics} abbreviated
$e^{\mathrm{rel}}_{\mathrm{int}}$ and $e^{\mathrm{rel}}_{\mathrm{ext}}$. On
$\mathcal{X}_{\mathrm{int}}$ the field is obtained directly from the solution
of~\eqref{eq:nystrom_linear_system}, whereas on $\mathcal{X}_{\mathrm{ext}}$ it is obtained from the representation
formula~\eqref{eq:representation_formula}.

\subsubsection{Selection of the transfer operator}\label{subsec:operator_selection}
We now compare the effectiveness of the two proposed preconditioner variants, dubbed Approach~I and Approach~II, so we can empirically select the best one going forward. Since both approaches yield virtually-identical discrete solutions, the selection rests solely on preconditioning performance, which we investigate under different physical and numerical settings.

We begin by comparing the spectral clustering of the unpreconditioned system~\eqref{eq:nystrom_linear_system} to that of the system preconditioned by each proposed approach. The top row of Figure~\ref{fig:ritz_spectrum} shows the full spectra of the matrices $\mathsf{L}$ (unpreconditioned) and $\mathsf{P}\mathsf{L}$ (for each of Approach~I and~II preconditioning), calculated for a dense construction at a coarse discretization level and a moderate wavenumber with $kh_C=1$ and $r_h=1$. Although all spectra show the accumulation at $\lambda=1$ expected of a second-kind integral equation discretization, the preconditioned spectra cluster far more tightly than the unpreconditioned one. Since a full spectrum computation at a more representative, finer resolution is prohibitively expensive with dense linear algebra, we instead compute, using the Arnoldi algorithm with reorthogonalization~\cite{daniel1976reorthogonalization}, the far cheaper Ritz values~\cite{trefethen2022numerical} using the matrix-free operators at $kh_C=0.125$ and $r_h=0.5$. The bottom row of Figure~\ref{fig:ritz_spectrum} shows that the clustering tightens further at finer resolutions. These Ritz values are the eigenvalues of the Hessenberg matrix generated by 80 Arnoldi iterations on each operator. In all cases, Approach~II shows a tighter clustering around $\lambda=1$ than Approach~I, and this will translate, as we show below, into reduced GMRES iteration counts.

\begin{figure}[htbp]
    \centering
    \includegraphics[width=\linewidth]{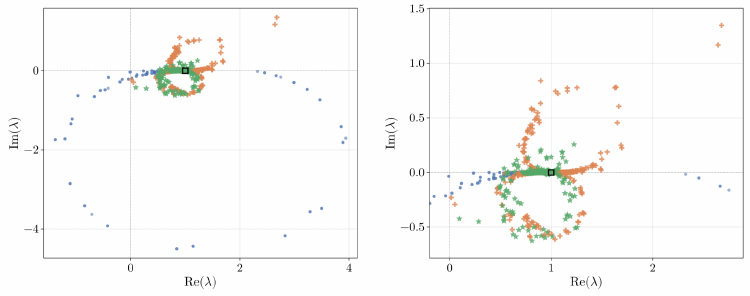}\\[4pt]
    \includegraphics[width=\linewidth]{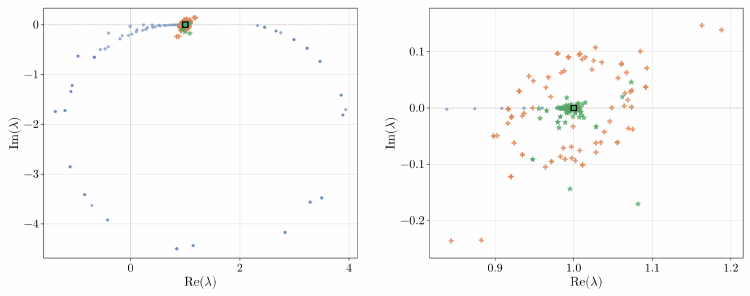}
    \caption{Spectral clustering around $\lambda=1$ for the unpreconditioned and preconditioned systems, for the benchmark constant-contrast disk with $k=10$, $\eta=1.25$. In every panel, disks mark the unpreconditioned spectrum, crosses Approach~I, asterisks Approach~II, and the open square marks $\lambda=1$. The left panel of each row spans the whole unpreconditioned spectrum, the right panel zooms around $\lambda=1$. Top row: full spectra from a dense construction at a coarse discretization, $kh_C=1$, $r_h=1$. Bottom row: 80 Ritz values from the matrix-free operators at a finer discretization, $kh_C=0.125$, $r_h=0.5$. These approximate eigenvalues resolve the outlying and clustered parts of the spectrum first, exposing the clustering without forming the matrices.}
    \label{fig:ritz_spectrum}
\end{figure}

Next, we use both approaches to solve the problem of time-harmonic plane-wave scattering by the constant-contrast disk.
Table~\ref{tab:bakeoff} reports that, at the same tolerance levels, the Approach~I preconditioned system solution uniformly needs more GMRES
iterations to converge in comparison with Approach~II. Remarkably, the number of iterations in Approach~II remain almost constant with respect to increasing wavenumbers and contrasts while those in Approach~I mildly increase. It should also be mentioned that although Approach~II is more frequency robust, both approaches need significantly fewer iterations to converge and their iterations do not grow as those of an unpreconditioned solve (see reference values in Table~\ref{tab:bakeoff}). In view of the presented evidence and
based on conditioning, spectrum clustering and frequency robustness of the resulting preconditioned system, we select Approach~II as the default preconditioning strategy in what follows.

\begin{table}[htbp]
  \centering
  \caption{Frequency robustness of the proposed preconditioning approaches
  tested on the benchmark constant-contrast disk at the Cartesian resolution $k\,h_C=0.125$, measured in GMRES iteration counts for different wavenumbers and contrasts. (a)~Wavenumber sweep with $\eta=1.25$ and $r_h=1.0$.
  (b)~Contrast sweep with $k=10$, $\eta$ and a triangulation size matched to the interior 
  wavelength $h=h_C/\sqrt{1+\eta}$ to resolve $k_{\mathrm{int}}=k\sqrt{1+\eta}$. In each 
  case the iteration count needed for the unpreconditioned solution convergence is also reported.  The FMM and GMRES relative residual tolerances for this experiment were both $10^{-8}$.}
  \label{tab:bakeoff}
  \medskip

  \begin{minipage}[t]{0.44\linewidth}
    \centering
    (a) Wavenumber ($\eta=1.25$, $r_h=1.0$)\par\smallskip
    \begin{tabular}{rrrr}
      \hline
      $k$ & iters (I) & iters (II) & iters (unpr.) \\
      \hline
      $10$ & $8$  & $7$ & 86 \\
      $20$ & $9$  & $6$ & 300 \\
      $40$ & $12$ & $7$ & $>600$ \\
      \hline
    \end{tabular}
  \end{minipage}\hfill
  \begin{minipage}[t]{0.55\linewidth}
    \centering
    (b) Contrast ($k=10$)\par\smallskip
    \begin{tabular}{rrrrr}
      \hline
      $\eta$ & $k_{\mathrm{int}}$ & iters (I) & iters (II) & iters (unpr.) \\
      \hline
      $1.25$ & $15.0$  & $9$  & $7$ & $85$  \\
      $2.5$  & $18.7$  & $13$ & $8$ & $144$ \\
      $5$    & $24.5$  & $16$ & $8$ & $241$ \\
      \hline
    \end{tabular}
  \end{minipage}
\end{table}

With Approach~II selected, we close this section by verifying that the resulting VDIM-based preconditioned solver exhibits the high-order convergence of the underlying VDIM discretization~\eqref{eq:linear_system}. Again on the disk with $k=10$ and $\eta=1.25$, we solve at VR interpolation orders $p\in\{1,2,3\}$ over $11$ shared triangulation sizes with $h\in[0.010,\,0.075]$, measuring the relative maximum-norm error~\eqref{eq:error_metrics} against a Fourier--Bessel series truncated so as to be accurate to a relative error of $10^{-14}$. The error is measured on a mesh-independent target set
$\mathcal{X}^{50}_{1/2}:=\Bigl\{\frac12\bigl(\cos\tfrac{2\pi j}{50},\,\sin\tfrac{2\pi j}{50}\bigr):\ j\in\{0,\ldots,49\}\Bigr\}$.
The error then probes the transmitted field, evaluated
through the representation formula~\eqref{eq:representation_formula}. The measured orders are
$\approx4.0,\,5.1,\,6.1$, matching those reported in~\cite{anderson2024fast} and confirming that
the preconditioner leaves the high-order convergence of VDIM intact.
Figure~\ref{fig:convergence_disk} shows these errors together with reference rates; we briefly remark, as mentioned also in~\cite{anderson2024fast}, that solely at $p = 1$ we observe convergence rate one order higher compared to that predicted by the theory presented in~\cite{anderson2024fast}.

\begin{figure}[ht]
    \centering
    \includegraphics[width=0.5\linewidth]{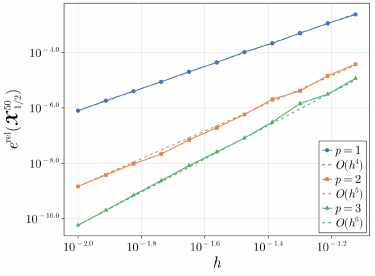}
    \caption{High-order convergence rates when using VR quadratures supporting $p\in\{1,2,3\}$ VDIM interpolation orders, measured on the 
    benchmark constant-contrast disk with $k=10$ and $\eta=1.25$ preconditioned with Approach~II.
    Results obtained from the relative maximum-norm error~\eqref{eq:error_metrics} on the interior target-point set
    $\mathcal{X}^{\,50}_{1/2}$.
    Dashed references: $\mathcal{O}(h^4)$, $\mathcal{O}(h^5)$, and $\mathcal{O}(h^6)$.}
    \label{fig:convergence_disk}
\end{figure}

\subsubsection{Comparison with a purely Cartesian discretization}\label{subsec:cartesian_comparison}
This section presents a comparison of the proposed preconditioned Nystr\"om solver, built on a triangular mesh fitted to the scatterer, with a purely Cartesian one---the structured-grid discretization of~\cite{duan2009high} preconditioned by the sparsifying preconditioner of~\cite{ying2015sparsifying}---in both preconditioner effectiveness and solution accuracy. To do so, we again solve the plane-wave scattering by a constant-contrast disk problem at different wavenumbers and contrast values. The Cartesian solver discretizes the Lippmann--Schwinger equation directly on a uniform grid over the $[-2,2]^2$ box with $kh_C=0.125$, yielding the system~\eqref{eq:LS_eq_on_cartesian_points}, which is  preconditioned by the sparsifying preconditioner~$\mathsf{S}^{(C)}$. Table~\ref{tab:struct_vs_fitted} compares the results to those of the proposed solver. Regarding GMRES iteration count, both preconditioned systems remain  within the same $6$--$10$ range, while the unpreconditioned count grows as the wavenumber or the contrast increases. The two discretizations offer strikingly different accuracies. As expected, because the structured discretization cannot resolve the contrast jump discontinuity across the boundary of the disk, its error stalls at around $10^{-2}$ to $10^{-1}$. The fitted discretization is instead six to seven orders of magnitude more accurate, reaching $10^{-9}$ to $10^{-7}$, where it is limited by the GMRES and FMM tolerances rather than by the discretization itself.

\begin{table}[htbp]
  \centering
  \caption{Fitted versus structured discretization, with (a) a wavenumber and (b) a contrast
  sweep. Columns \emph{fit.}\ and \emph{unp.} are the $\mathsf{P}$-preconditioned and the
  unpreconditioned Nystr\"om system~\eqref{eq:nystrom_linear_system}, \emph{str.}\ the
  $\mathsf{S}^{(C)}$-preconditioned Cartesian system~\eqref{eq:LS_eq_on_cartesian_points}, all at
  the same stopping tolerance. Errors are relative maximum-norm~\eqref{eq:error_metrics} against
  the Fourier--Bessel reference, each evaluated on the interior discretization points of the respective solver:
  $e_{\mathrm{fit}}$ on the quadrature set~\eqref{eq:measurement_sets} and
  $e_{\mathrm{struct}}$ on the interior Cartesian nodes~\eqref{eq:interior_cartesian_indices}.  The FMM and GMRES relative residual tolerances for this experiment were both $10^{-8}$.}
  \label{tab:struct_vs_fitted}
  \medskip
  \setlength{\tabcolsep}{2.5pt}

  \begin{minipage}[t]{0.48\linewidth}
    \centering
    (a) Wavenumber ($\eta=1.25$)\par\smallskip
    \begin{tabular}{rrrrrrr}
      \hline
      & & \multicolumn{3}{c}{iters} & \multicolumn{2}{c}{Max. error} \\
      $k$ & $k_{\mathrm{int}}$ & fit. & str. & unp. & $e_{\mathrm{struct}}$ & $e_{\mathrm{fit}}$ \\
      \hline
      $5$  & $7.5$ & $6$ & $6$ & $28$  & $2.1\times10^{-2}$ & $2.7\times10^{-9}$ \\
      $10$ & $15$  & $7$ & $6$ & $85$  & $8.7\times10^{-2}$ & $2.1\times10^{-9}$ \\
      $20$ & $30$  & $6$ & $6$ & $298$ & $3.7\times10^{-2}$ & $1.2\times10^{-8}$ \\
      $40$ & $60$  & $7$ & $7$ & $>600$   & $2.1\times10^{-2}$ & $1.4\times10^{-7}$ \\
      \hline
    \end{tabular}
  \end{minipage}\hfill
  \begin{minipage}[t]{0.48\linewidth}
    \centering
    (b) Contrast ($k=10$)\par\smallskip
    \begin{tabular}{rrrrrrr}
      \hline
      & & \multicolumn{3}{c}{iters} & \multicolumn{2}{c}{Max. error} \\
      $\eta$ & $k_{\mathrm{int}}$ & fit. & str. & unp. & $e_{\mathrm{struct}}$ & $e_{\mathrm{fit}}$ \\
      \hline
      $1.25$ & $15$    & $7$  & $6$ & $85$  & $8.7\times10^{-2}$ & $2.1\times10^{-9}$ \\
      $2.5$  & $18.71$ & $8$  & $6$ & $144$ & $1.1\times10^{-1}$ & $1.2\times10^{-8}$ \\
      $5$    & $24.49$ & $8$  & $6$ & $241$ & $8.1\times10^{-2}$ & $4.2\times10^{-8}$ \\
      $10$   & $33.17$ & $10$ & $7$ & $415$ & $1.2\times10^{-1}$ & $1.5\times10^{-7}$ \\
      \hline
    \end{tabular}
  \end{minipage}
\end{table}

\subsubsection{Parameter dependence: mesh sizes, frequency and contrast}\label{subsec:parameter_dependence}

In this section we study the effect of the relevant parameters on the preconditioner's effectiveness and the overall accuracy of the solver.  We begin by examining the effect of the mesh sizes $h$ and $h_C$ on accuracy and performance.
Figure~\ref{fig:error_heatmap} reports results from sweeping four Cartesian mesh sizes against five triangulation sizes for the benchmark disk with $k=10$ and $\eta=1.25$. Along each row, at fixed $h_C$, the iteration
count is unchanged while the error decreases with $h$ at the rate determined by the order of the VDIM scheme. Along each column, at fixed $h$, the error is unchanged while the iteration count decreases with $h_C$. As expected then, accuracy is therefore governed by $h$ and preconditioning performance by $h_C$, and refining both at a fixed $r_h$ improves the two together, independently of the problem size. Interestingly, the behavior along a row persists well beyond the range spanned by the figure. At the finest Cartesian resolution, $kh_C=0.125$, sweeping $r_h$ from $3$ down to $0.25$---that is, from
$3.2\times10^{4}$ to $4.5\times10^{6}$ quadrature nodes---leaves the iteration count constant. Meanwhile $e^{\mathrm{rel}}_{\mathrm{int}}$ decreases from $1.4\times10^{-6}$ to a solver/FMM floor near $5\times10^{-10}$, first reached at $r_h\approx0.5$; $e^{\mathrm{rel}}_{\mathrm{ext}}$
follows it, decreasing from $8.1\times10^{-7}$ to $3\times10^{-10}$.

\begin{figure}[htbp]
    \centering
    \includegraphics[width=\linewidth]{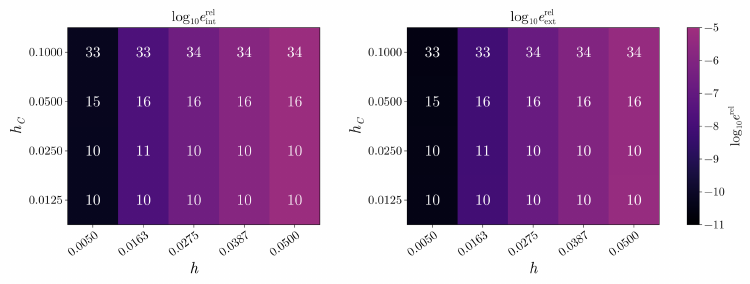}
    \caption{Relative maximum error heatmap matrix overlaid with GMRES iteration count
    at different mesh sizes $h$ and $h_C$ computed for the benchmark problem of 
    plane-wave scattering by a constant-contrast disk with $k=10$, $\eta=1.25$ and preconditioned
    using Approach~II. The colormap reports the magnitude of the relative maximum
    norm~\eqref{eq:error_metrics} over $\mathcal{X}_{\mathrm{int}}$ (left) and
    $\mathcal{X}_{\mathrm{ext}}$ (right), both defined in~\eqref{eq:measurement_sets}, and the
    overlaid numbers on each cell report the number of GMRES iterations needed for convergence at
    the corresponding mesh sizes. The FMM and GMRES relative residual tolerances for this experiment were both $10^{-12}$.}
    \label{fig:error_heatmap}
\end{figure}

We now fix the resolution and sweep $k\in\{5,10,20,40\}$, repeating at four resolutions $kh_C\in\{1,\tfrac12,\tfrac14,\tfrac18\}$, i.e., $6.3$, $12.6$, $25.1$ and $50.3$ points per wavelength, with $\eta=1.25$ and $h=h_C/\sqrt{1+\eta}$. Figure~\ref{fig:coarse_ppw} shows
the resulting sweeps of the interior error and of the iteration count at each fixed resolution. Both quantities grow more slowly with $k$ as the resolution is refined, and essentially become flat at $50.3$ points per wavelength. At that resolution the method appears to be frequency independent
both in accuracy and in the number of iterations.

\begin{figure}[htbp]
    \centering
    \includegraphics[width=\linewidth]{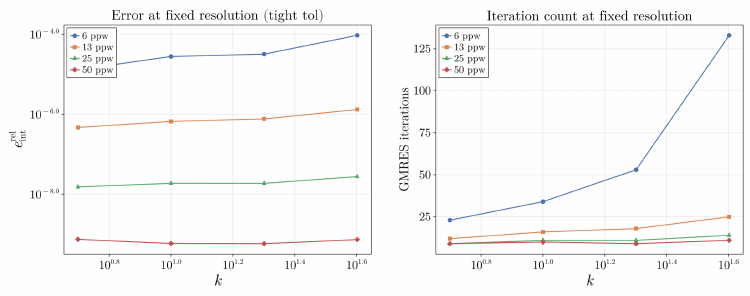}
    \caption{Interior error and GMRES iteration counts for different wavenumbers at fixed
    resolution, i.e. constant $k\,h_C$ (equiv. $k_{\mathrm{int}}\,h$) for the constant-contrast
    disk plane-wave scattering problem with $\eta=1.25$. One curve per Cartesian resolution: relative maximum-norm
    interior error vs.\ $k$ at fixed resolution (left) and preconditioned GMRES iteration count
    over the same sweeps (right).  The FMM and GMRES relative residual tolerances for this experiment were $10^{-14}$ and $10^{-12}$, respectively.}
    \label{fig:coarse_ppw}
\end{figure}

At this finest resolution, $kh_C=0.125$, we further test robustness to increasing contrast $\eta$, keeping $h=h_C/\sqrt{1+\eta}$ so the triangulation continues to resolve the interior wavelength $k_{\mathrm{int}}=k\sqrt{1+\eta}$. Table~\ref{tab:robustness} lists the underlying wavenumber
sweep (panel (a), matching the finest curve of Figure~\ref{fig:coarse_ppw}) and the contrast sweep (panel (b)). The preconditioner is robust to both, with the interior error staying on the $5$--$8\times10^{-10}$ floor throughout. The only exception is the iteration count at $\eta=10$
in panel (b), which rises to $15$. This is a resolution effect, not a failure of the preconditioner: since $h_C$ is tied to the exterior wavenumber $k$, it resolves the interior field progressively worse as the contrast grows, while the error, governed by $h$, is unaffected. 


\begin{table}[htbp]
  \centering
  \caption{Constant-contrast disk solved with Approach~II at $p=2$, with $kh_C=0.125$ and $h=h_C/\sqrt{1+\eta}$. Preconditioned GMRES iteration count and interior relative maximum-norm error against the Fourier--Bessel reference. Panel (a) is a wavenumber sweep at fixed contrast and panel (b) a contrast sweep at fixed wavenumber. The $k=10$ row of (a) and the $\eta=1.25$ row of (b) are the same problem discretized on two independently generated meshes. The FMM and GMRES relative residual tolerances for this experiment were $10^{-14}$ and $10^{-12}$, respectively.}
  \label{tab:robustness}
  \medskip

  \begin{minipage}[t]{0.48\linewidth}
    \centering
    (a) Wavenumber ($\eta=1.25$)\par\smallskip
    \begin{tabular}{rrrrr}
      \hline
      $k$ & $k_{\mathrm{int}}$ & $N_Q$ & iters & $e^{\mathrm{rel}}_{\mathrm{int}}$ \\
      \hline
      $5$  & $7.5$ & 157\,734     & $9$  & $7.38\times10^{-10}$ \\
      $10$ & $15$  & 628\,812     & $10$ & $5.81\times10^{-10}$ \\
      $20$ & $30$  & 2\,511\,312  & $9$  & $5.75\times10^{-10}$ \\
      $40$ & $60$  & 10\,038\,072 & $11$ & $7.27\times10^{-10}$ \\
      \hline
    \end{tabular}
  \end{minipage}\hfill
  \begin{minipage}[t]{0.48\linewidth}
    \centering
    (b) Contrast ($k=10$)\par\smallskip
    \begin{tabular}{rrrrr}
      \hline
      $\eta$ & $k_{\mathrm{int}}$ & $N_Q$ & iters & $e^{\mathrm{rel}}_{\mathrm{int}}$ \\
      \hline
      $1.25$ & $15$    & 628\,812    & $10$ & $5.81\times10^{-10}$ \\
      $2.5$  & $18.71$ & 978\,606    & $11$ & $5.03\times10^{-10}$ \\
      $5$    & $24.49$ & 1\,676\,580 & $11$ & $5.28\times10^{-10}$ \\
      $10$   & $33.17$ & 3\,072\,084 & $15$ & $6.07\times10^{-10}$ \\
      \hline
    \end{tabular}
  \end{minipage}
\end{table}

\subsubsection{Cost and computational scaling}\label{subsec:scaling}

In this section we report the computational cost of our overall preconditioned method and its scaling with the problem size. 
We measure both the wall-clock time and the storage needed to solve the preconditioned system with our FMM-accelerated forward-operator application, at different problem sizes determined by the total number of quadrature nodes $N_Q$, which is increased in two ways: first, by refining
the mesh at fixed wavenumber $k=10$; second, by sweeping different wavenumbers while keeping the resolution fixed at $k\,h_C=0.125$. 

Figure~\ref{fig:scaling} (left) shows the total solve time for both sweeps and its expected linearithmic scaling, while Table~\ref{tab:scaling} breaks down the $k$-sweep experiment times by individual stages of the solver pipeline, excluding mesh generation and the reference-solution evaluation. All runs were parallelized with a number of threads matching the system's $16$ physical cores, using OpenMP threads for the FMM and Julia threads for the transfer-operator assembly. Figure~\ref{fig:scaling} (right) and Table~\ref{tab:scaling} show memory usage growing linearly in $N_Q$. The largest problem we tested, at $k=40$ and $\eta=1.25$, needed over $10$ million degrees of freedom, $5.1$~GiB of storage for the operators, and $19.6$~GiB of peak memory usage (the gap being attributable largely to usage by the FMM), and was solved end-to-end in under $7$ minutes.

\begin{figure}[htbp]
    \centering
    \includegraphics[width=\linewidth]{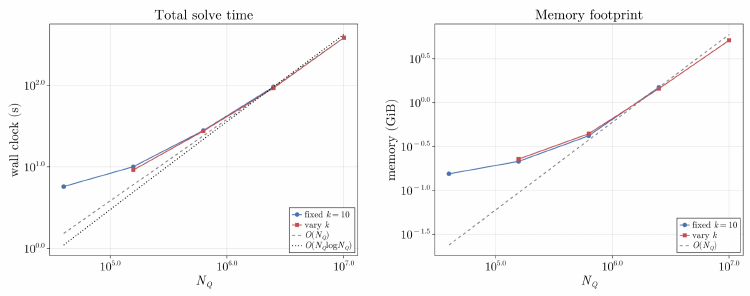}
    \caption{Computational scaling of the preconditioned method with Approach~II for the
    constant-contrast disk benchmark problem.
    Left: total wall time vs $N_Q$ (log--log). Right: persistent live-heap memory
    vs $N_Q$ (log--log). Blue: mesh refinement at
    fixed $k=10$. Red: wavenumber sweep, $hk$ fixed.}
    \label{fig:scaling}
\end{figure}

\begin{table}[htbp]
  \centering
  \caption{Constant-contrast disk with $\eta=1.25$ and $k\,h_C=0.125$, solved in parallel with $16$ threads.
  Per-stage wall-clock times, where $t_{\mathcal{V}_{\Omega}}$ is the
  volume-potential assembly at the quadrature nodes, $t_{\mathrm{pre}}$ the preconditioner and
  transfer-operator build, and $t_{\mathrm{solve}}$ the preconditioned GMRES solve, together with
  their sum $t_{\mathrm{tot}}$, the persistent
  memory (precomputed operator storage) and the peak memory (solve stage) usage.  The FMM and GMRES relative residual tolerances for this experiment were both $10^{-8}$.}
  \label{tab:scaling}
  \begin{tabular}{rrrrrrrrr}
    \hline
    $k$ & $N_Q$ & iters & $t_{\mathcal{V}_{\Omega}}$ & $t_{\mathrm{pre}}$ & $t_{\mathrm{solve}}$ & $t_{\mathrm{tot}}$ (s) & mem (GiB) & peak (GiB) \\
    \hline
    5  & 157\,734     & 6 & 4.9   & 1.47  & 2.8   & 9.2   & 0.23 & 5.10  \\
    10 & 628\,812     & 7 & 12.7  & 6.05  & 8.7   & 27.5  & 0.44 & 13.81 \\
    20 & 2\,511\,312  & 6 & 40.8  & 22.52 & 29.4  & 92.7  & 1.44 & 17.22 \\
    40 & 10\,038\,072 & 7 & 157.6 & 88.56 & 137.5 & 383.7 & 5.13 & 19.58 \\
    \hline
  \end{tabular}
\end{table}

It follows from these results that the cost of a single GMRES iteration is proportional to the problem size. From Table~\ref{tab:scaling}, the solve time per iteration and per degree of freedom is nearly constant---$1.98$, $1.95$, $1.96\,\mu$s---over a sixteen-fold range in $N_Q$, of which $1.50\,\mu$s (at $k=20$) is the forward-map application (almost entirely one FMM) and the rest is the preconditioner application and Krylov orthogonalization. Thus, the total solve time is, to an exceedingly good approximation, the iteration count times a size-independent constant times $N_Q$, justifying the use of iteration count as the primary figure of merit earlier in this section. Clearly, preconditioning pays for itself several times over: at $\eta=1.25$, $k\,h_C=0.125$, the unpreconditioned solve takes $28$, $85$, $298$ GMRES iterations at $k=5,10,20$ (vs.\ $6$, $7$, $6$ preconditioned) and $11.5$, $91.2$, $1300.9$~s (vs.\ $2.8$, $8.7$, $29.4$~s). Even charging the one-time preconditioner and transfer-operator assembly in full to the preconditioned run---$1.5$, $6.1$, $22.5$~s, i.e.\ $4.3$, $14.8$, $51.9$~s total---it remains $25\times$ faster at $k=20$, and increasingly so as $k$ grows. (The volume-potential assembly, common to both, is naturally excluded from analysis here.)

\subsection{More general examples}\label{subsec:challenging}

This section explores the limits of the proposed solver on several challenging scattering scenarios. Throughout, we use Approach~II for the quadrature-to-Cartesian transfer operator $\mathsf T_Q^{C}$, and take the nodes from a VR  quadrature of interpolation order $p=2$.

Our first example concerns plane-wave scattering by an inhomogeneous scatterer supported on a kite-shaped domain~\cite[sec.~3.5]{COLTON:2012}, whose contrast $m\equiv 1-n$ (defined in~\eqref{eq:contrast_def}) is linearly graded and has a jump discontinuity across the smooth boundary. Superimposed on this is a sharply localized smooth bump in the refractive index (see Figure~\ref{fig:kite_lens_refined}). The resulting spatially varying contrast is given by $m(x)=1.5\,(x_2-1.1)-21.6\,w(x)$, where $w\in C_0^\infty$ is a radial window function that vanishes outside the disk $\{x\in\R^2:|x-(0.4,-0.5)|\le 0.18\}$. This contrast has two interesting features. On the one hand, it changes sign across a curve near the top of the kite-shaped domain (the level set $x_2=1.1$, away from the bump), and, on the other, the localized bump creates a small region where the contrast peaks at $\max_{\overline\Omega}|m|=24$. Notably, resolving this high-contrast region is straightforward with the proposed solver. Because our fitted discretization places its quadrature points on an unstructured triangulation, we can easily refine the mesh locally.

The top row of Figure~\ref{fig:kite_lens_refined} shows the contrast field (left) and a close-up of the locally refined triangulation, with each element colored by its size (right). The bottom row shows the real part of the total field obtained from the preconditioned system for a plane wave propagating in the direction $d=(1,0)$ with wavenumber $k=75$ at $k\,h_C=0.25$. The bottom-left panel shows the full computational domain and the scattering produced by the sign-changing contrast, while the bottom-right panel zooms into the high-contrast region, illustrating how the discretization handles two disparate wavelength scales seamlessly. This solution was computed with a single refinement level $r_h=1$ ($h=3.33\times10^{-3}$, locally reduced to $h_{\mathrm{peak}}=6.67\times10^{-4}$ in the refined zone), amounting to $N_Q=6.6\times10^{6}$ quadrature nodes, for which the preconditioned solver required $26$ GMRES iterations to reach a relative residual tolerance of $10^{-9}$.

\begin{figure}[htbp]
    \centering
    \includegraphics[width=1.0\linewidth]{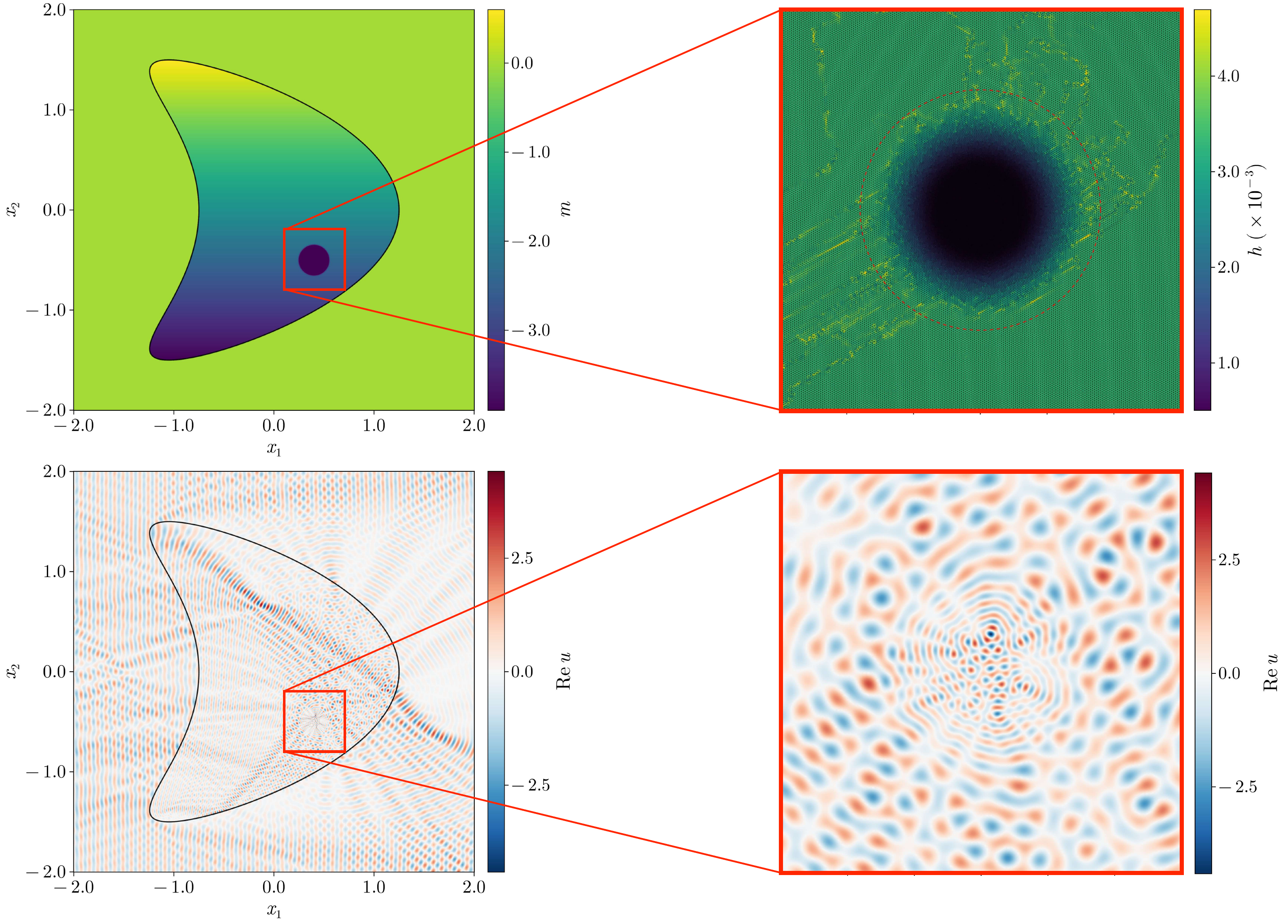}
    \caption{Planewave scattering problem by a scatterer modeled by a graded contrast 
    supported on a kite-shaped domain with a concentrated
    high magnitude region. The solution is computed for an incident field wavenumber of $k=75$
    and preconditioned with Approach~II.
    Top left: contrast $m$, with the color range clipped to the range $[-3.9,0.6]$ of the graded part of the index; note that the contrast peak is approximately $6$ times larger than the color range limit.
    Top right: local mesh size $h$ in the refined zone (dashed circle: peak support, radius $0.18$).
    Bottom left: total field $\mathrm{Re}\,u$. Bottom right: zoom of $\mathrm{Re}\,u$ near the contrast peak.}
    \label{fig:kite_lens_refined}
\end{figure}

In the next example we consider a scatterer supported on the unit disk, with a ``yin-yang'' type refractive index depicted in Figure~\ref{fig:yinyang} (top-left). The disk is split by a smooth curved interface into two subregions, $\Omega_1$ and $\Omega_2$, in which the refractive index takes the base values $4$ (dark gray) and $16$ (white). To each subregion we add a Gaussian bump, the two bumps having opposite signs but equal amplitude, so that the refractive index varies smoothly within each subregion. The resulting contrast is piecewise smooth, with a jump discontinuity across the curved interface. 

As described in Section~\ref{sec:Nystrom_discretization}, we discretize the integration domain $\Omega=\Omega_1\cup\Omega_2$ by meshing each disjoint region separately and assembling the volume potential as in~\eqref{eq:vol_pot_decomp} with $N_d=2$, so that the polynomial interpolant $F^{(\tau)}$ and the boundary correction $\mathcal{B}_{\Omega_\ell}$ of~\eqref{eq:DIM_decomp} are always confined to a single region, with the internal interface entering as part of the boundary $\partial\Omega_\ell$ of each subdomain rather than as a discontinuity interior to the integration domain. We highlight the fact that, although each triangulation conforms to the internal interface, which is resolved by mesh edges on both of its sides, by such a construction the two triangulations need not match node-to-node across it. The remaining panels of Figure~\ref{fig:yinyang} show the real part (top right, with a zoomed view in the bottom right) and the absolute value (bottom left) of the total field produced by an incident plane wave with $k=64$ propagating from below, obtained by solving the preconditioned system in $29$ GMRES iterations.

\begin{figure}[htbp]
    \centering
      \includegraphics[width=0.9\textwidth]{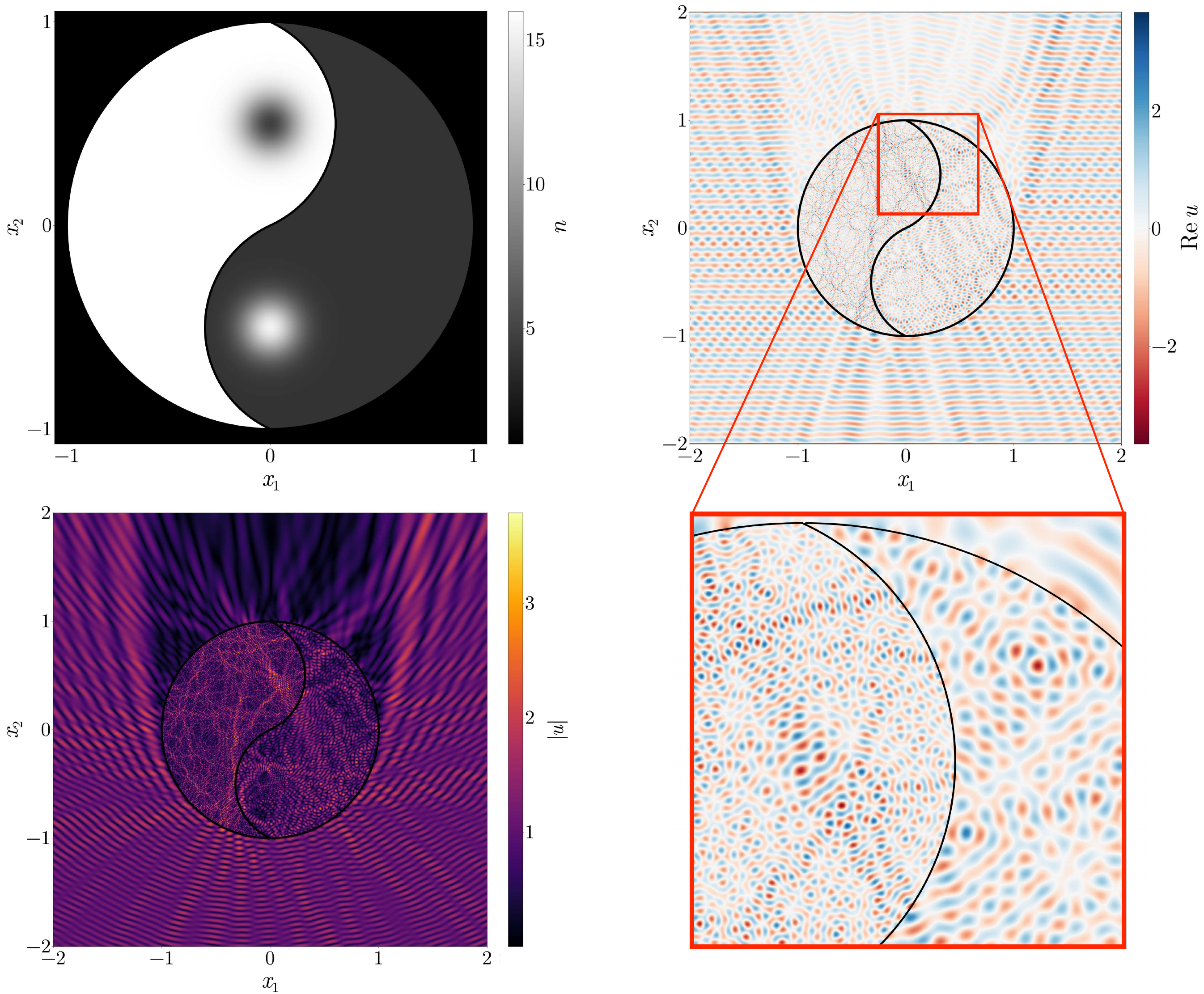}
    \caption{Plane-wave scattering by the yin-yang contrast field (top left),
    which has a jump discontinuity across the curved interface. The right
    panels show the real part of the total field obtained from the
    Approach~II preconditioned problem, for plane-wave incidence with $k=64$
    propagating from below. The bottom-right panel is a zoomed view of the
    top-right one, and the bottom-left panel shows the absolute value of the
    total field. The solver required 29 iterations to reach an absolute
    residual tolerance of $10^{-6}$. A convergence study at $k=8$, using the
    solution values on a set $\mathcal{X}^{100}_{1.2}$ of $100$ points on a ring at
    distance $0.2$ from the obstacle, gave a convergence order of $4.34$ in
    the relative maximum-pointwise norm~\eqref{eq:error_metrics}.}
    \label{fig:yinyang}
\end{figure}

Our next example concerns plane-wave scattering by an open-cavity resonator. In detail, we take as support of the contrast function the truncated annulus $\Omega=\{x\in\R^2:\,r<|x|<R,\ |\theta|>\theta_0\}$, with $\theta$ the polar angle and the values $(r,R,\theta_0)=(0.8,\,1,\,20^\circ)$. Geometrically, $\Omega$ is a thin annular shell from which a wedge of angular width $2\theta_0$ has been removed (see the left panel in Figure~\ref{fig:open_resonator}). Within $\Omega$ we take the radial contrast $m(x)=\exp(-((\rho-0.9)/0.1201)^2)$, $\rho=|x|$, extended by zero outside. The total field solution of the right panel in Figure~\ref{fig:open_resonator} corresponds to the plane wave $u^{\inc}(x)=\e^{ik\,x\cdot d}$ with $d=(-1,0)$, which enters the cavity through its aperture, at $k=200$. With $kh_C=0.25$ and $r_h=3$, that is, $8.4$ elements per wavelength, the preconditioned solver reached a relative residual of $3.15\times10^{-10}$ in $13$ GMRES iterations, which are of the order of those of the constant refractive index benchmark problem, but here on a geometry that traps the field. The value $r_h=3$, coarser than in the preceding examples, is used due to memory limitations rather than accuracy. 

At $k=200$ a mesh refinement to assess the error is not affordable. We therefore, as in the yin-yang example, carry out a self-convergence study at the lower frequency $k=8$, measuring the relative maximum-norm error~\eqref{eq:error_metrics} of each level against the finest one, taken as reference, on $1000$ points uniformly distributed on the part of a circle of radius $0.9$ lying inside the C-shaped scatterer. Refining the triangulation through $r_h=h/h_C\in\{0.7,0.5,0.3,0.1\}$ with $kh_C=0.25$ fixed yields a fitted order of $4.97$, while refining through $r_h\in\{1.0,0.8,0.6,0.4,0.2\}$ with $kh_C=0.125$ yields a fitted order of $5.01$. The two observed orders agree closely, exhibiting high-order convergence even in this geometrically delicate region of the scatterer, while the GMRES iteration count stays flat at $9$ over the whole span in problem size. At $k=200$, as in the figure, a reference-free check is still available for this example. Since the support, the contrast and the incidence direction are all invariant under reflection with respect to the $x_1$ axis, so is the exact field. The unstructured triangulation, however, does not respect this symmetry, yet the computed solution differs from its own reflection by only $1.2$--$1.6\times10^{-3}$ relative to $\max|u|$ at points inside and outside the scatterer, including at the cavity opening and around the corners. 

\begin{figure}[htbp]
    \centering
    \includegraphics[width=\linewidth]{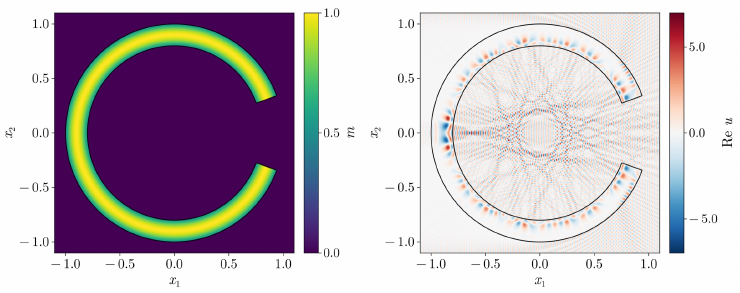}
    \caption{Planewave scattering by an open resonator.
    The left panel shows the contrast while the right panel shows the real part of the total field for a right-to-left incident plane wave, $d=(-1,0)$, with $k=200$ at $kh_C=0.25$ and $r_h=3$, computed with the preconditioned solver in $13$ GMRES iterations.}
    \label{fig:open_resonator}
\end{figure}

We conclude this section by reproducing and extending the photonic-crystal example of~\cite[Sec.~9.2]{StrauszerCaussade2023windowed}, which considers a plane wave at normal incidence, in transverse
electric (TE) polarization, on a two-dimensional crystal formed by a lattice of low-index pores, at
a frequency inside the stop band. The lattice geometry, the pore radius $r_p$, the incident
wavelength and the non-dimensionalization are those of that reference, to which we refer the reader
for all details. We take the pores as the scatterer and the surrounding crystal medium as the
background, $n\equiv1$. Unlike the solution in~\cite{StrauszerCaussade2023windowed}, which is obtained from a boundary integral formulation for an infinite periodic crystal, we here solve with our preconditioned volume solver a crystal of finite thickness and finite length, comprising $N_x$ periods.
Moreover, we consider an imperfect crystal, meant to model fabrication defects, in which the refractive index is given by $n(\mathbf{x})=n_0+\tilde n(\mathbf{x})$, with $n_0=1/2.6$ and $\tilde n$ being a normalized superposition of $500$ random Gaussian bumps of widths uniformly distributed in $[0.5r_p,2r_p]$ and amplitudes limited by $\max|\tilde n|=\delta\,n_0$, so that the dimensionless parameter $\delta$ is the maximum relative defect of the pore index (Figure~\ref{fig:g3g5_strauszer_imperfect}(a) shows the case $\delta=0.1$). The resulting contrast is smooth and non-constant inside each curved pore and can therefore no longer be treated by a boundary integral formulation.
Solving at $kh_C=0.125$, we sweep the crystal length over $N_x\in\{11,21,31\}$ and the defect over $\delta\in\{0.1,0.2\}$. The incident field is almost completely reflected in every case, within the same region as in~\cite[Fig.~10]{StrauszerCaussade2023windowed}, confirming that the tested frequency lies in the stop band. GMRES iterations remain essentially flat, between $9$ and $11$, over problem sizes ranging from $1.6$ to $6.8$ million degrees of freedom. Figure~\ref{fig:g3g5_strauszer_imperfect}(b) shows the real part of the total field for $N_x=21$ and $\delta=0.1$.

\begin{figure}[htbp]
    \centering
    \subfloat[Imperfect contrast field $\tilde n/n_0$, $\delta=0.1$]{\includegraphics[width=0.92\linewidth]{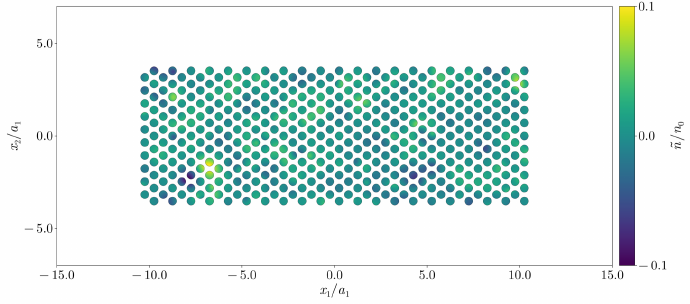}}\\
    \subfloat[Present solver, imperfect crystal, $\delta=0.1$]{\includegraphics[width=0.92\linewidth]{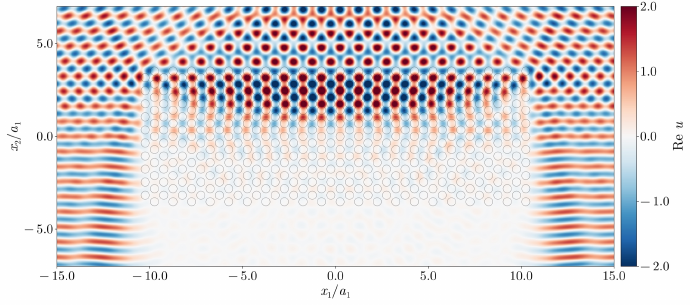}}
    \caption{Finite imperfect photonic-crystal slab with $N_x=21$ periods at the
    TE stop-band frequency. Panel (a) shows the smoothly varying pore-index
    perturbation $\tilde n/n_0$ and (b) the resulting total field $\mathrm{Re}\,u$.
    Both panels span the same window $x_1/a_1\in[-15,15]$, $x_2/a_1\in[-7,7]$, so
    the entire slab and the surrounding free space are visible.}
    \label{fig:g3g5_strauszer_imperfect}
\end{figure}

\section{Conclusion}\label{sec:conclusion}
This paper has demonstrated a successful marriage between the high-order
discretization of volume integral equations over unstructured meshes with a
class of fast solvers (here, the sparsifying preconditioner) that exploit
structured meshes. Theoretical guarantees provide sufficient conditions under
which the proposed preconditioner is invertible; in practice, we have never
observed non-invertibility. The results of the preconditioner are in some sense
ideal: we have demonstrated that the iteration counts under coupling can be
made roughly equal to those required for a structured-grid discretization
preconditioned with the fast solver, with costs dominated by the application of
the unpreconditioned forward map. The preconditioned spectrum is tightly
clustered at $\lambda = 1$. Our ongoing work in this area concerns extension to
3D and to other volume integral equation formulations arising in acoustics and
electromagnetics; applications of the methods in electrical engineering and
inverse problems are being pursued.

\section*{Acknowledgments}
The authors are grateful to Leonardo Zepeda-N\'{u}\~{n}ez for the public
availability of a Julia-language implementation of the sparsifying
preconditioner and for its revival for modern Julia, upon which code the solver
of this paper is based. The first author acknowledges support from the NSF
under awards DMS-2514012 and DMS-2231482.  The second and fourth authors
acknowledge support from the Dutch Research Council (NWO) under the project
``Density interpolation methods for the fast and high-order evaluation of volume
potentials in complex geometries'' with file number OCENW.M.22.387 of the
research programme Open Competition Domain Science (M).

\appendix
\section{Invertibility conditions in practice}\label{app:precond_validation}

This appendix presents numerical evidence to support some of the claims and assumptions made in Section~\ref{sec:invertibility} on the invertibility of the preconditioner.

Table~\ref{tab:invertibility_certificate} reports the quantities relevant to the invertibility analysis of Approach~I. That analysis relies on the factor $\|\mathsf{Z}_{\mathcal{I}}\|_\infty$, i.e., the largest absolute row sum of $\mathsf{I}_{N_C}-\mathsf{R}^{(C)}$ over the interior Cartesian rows $\mathcal{I}$~\eqref{eq:interior_cartesian_indices}, decreasing as $r_h\to0$. For Approach~I it does: Table~\ref{tab:invertibility_certificate} shows it decreasing linearly in the mesh ratio $r_h$, confirming the first-order rate established in Lemma~\ref{lem:approachI_estimate}.
For Approach~II, in turn, the same quantity plateaus near $1.11$ as $r_h\to0$, so no amount of refinement brings the round trip $\mathsf{R}^{(C)}_{\mathrm{II}}$ close to the Cartesian identity at the interior Cartesian nodes. Approach~II therefore admits no perturbative (Neumann series) analysis of the kind carried out for Approach~I, which is why its invertibility relies instead on the assumed condition~\eqref{eq:diss_on_range} on $\mathsf{S}^{(C)}$ (Theorem~\ref{thm:P_invertible_unconditional}).
Lemma~\ref{lem:Q_invertibility} also provides the criterion~\eqref{eq:invertibility_condition_unsplit}, namely $\|(\mathsf{S}^{(C)})^{-1}(\mathsf{S}^{(C)}-\mathsf{I}_{N_C})\mathsf{Z}\|_\infty<1$, whose
left-hand side Table~\ref{tab:invertibility_certificate} reports for $\eta=-m=1.25$ (in fact, it is proportional to $|\eta|$ in this case, so the corresponding values at $\eta=2.5$ are simply double those tabulated). As expected, it inherits
the behavior of $\|\mathsf{Z}_{\mathcal{I}}\|_\infty$, decreasing with $r_h$ for Approach~I,
plateauing for Approach~II. Over the range of
mesh ratios reported it falls below one, for Approach~I, at the three smallest values of $r_h$ considered, so this sufficient condition certifies invertibility in those cases; for Approach~II it remains near $6.25$ throughout. As an additional check for invertibility, Table~\ref{tab:invertibility_certificate} reports the smallest singular value of $\mathsf{P}$, whenever the memory budget allowed storing a dense instance of $\mathsf{P}$ and illustrates that $\mathsf{P}$ remains invertible in cases with $r_h\geq 1$ where the sufficient condition is not met---highlighting that the conditions of Lemma~\ref{lem:Q_invertibility} are sufficient but not necessary.

\begin{table}[htbp]
  \centering
  \caption{Preconditioner invertibility check for a constant-contrast disk with $k=10$ at the Cartesian resolution $k\,h_C=1$ and several mesh ratios $r_h$. Writing $\mathsf{Z}:=\mathsf{I}_{N_C}-\mathsf{R}^{(C)}$ and
  $\mathsf{Z}_{\mathcal{I}}$ for its restriction to the interior rows
  $\mathcal{I}$~\eqref{eq:interior_cartesian_indices}, the contrast-independent column
  $\|\mathsf{Z}_{\mathcal{I}}\|_\infty$ is the factor bounded in
  Lemma~\ref{lem:approachI_estimate}, while \emph{criterion} is the left-hand side
  of~\eqref{eq:invertibility_condition_unsplit},
  $\|(\mathsf{S}^{(C)})^{-1}(\mathsf{S}^{(C)}-\mathsf{I}_{N_C})\mathsf{Z}\|_\infty$ for $\eta=1.25$, which must
  fall below one to certify invertibility. Dashes mark the cases in which the memory budget did
  not allow forming a dense $\mathsf{P}$.}
  \label{tab:invertibility_certificate}
   \begin{tabular}{rrrrrrrrr}
    \hline
     & \multicolumn{4}{c}{Approach~I} & \multicolumn{4}{c}{Approach~II} \\
    \cmidrule(lr){2-5}\cmidrule(lr){6-9}
     & & & \multicolumn{2}{c}{$\sigma_{\min}(\mathsf{P})$}
     & & & \multicolumn{2}{c}{$\sigma_{\min}(\mathsf{P})$} \\
    $r_h$ & $\|\mathsf{Z}_{\mathcal{I}}\|_\infty$ & criterion & $\eta{=}1.25$ & $\eta{=}2.5$
          & $\|\mathsf{Z}_{\mathcal{I}}\|_\infty$ & criterion & $\eta{=}1.25$ & $\eta{=}2.5$ \\
    \hline
    $3.00$    & $2.22$   & $9.28$  & $8.1\times10^{-2}$ & $2.0\times10^{-2}$ & $1.55$ & $7.77$ & $1.2\times10^{-1}$ & $4.4\times10^{-2}$ \\
    $2.00$    & $1.38$   & $6.72$  & $8.6\times10^{-2}$ & $3.7\times10^{-2}$ & $1.29$ & $6.62$ & $1.3\times10^{-1}$ & $5.6\times10^{-2}$ \\
    $1.00$    & $0.662$  & $2.92$  & $4.8\times10^{-2}$ & $2.2\times10^{-2}$ & $1.19$ & $6.24$ & $1.3\times10^{-1}$ & $6.1\times10^{-2}$ \\
    $0.50$    & $0.329$  & $1.72$  & --                 & --                 & $1.13$ & $6.25$ & --                 & -- \\
    $0.25$    & $0.169$  & $1.09$  & --                 & --                 & $1.12$ & $6.27$ & --                 & -- \\
    $0.125$   & $0.0856$ & $0.545$ & --                 & --                 & $1.11$ & $6.26$ & --                 & -- \\
    $0.0625$  & $0.0432$ & $0.243$ & --                 & --                 & $1.11$ & $6.25$ & --                 & -- \\
    $0.03125$ & $0.0215$ & $0.112$ & --                 & --                 & $1.11$ & $6.24$ & --                 & -- \\
    \hline
  \end{tabular}
\end{table}

Next we consider the hypothesis of Theorem~\ref{thm:P_invertible_unconditional}, which underlies the invertibility of the Approach~II preconditioner. The quantity of concern is the normalized $\mathsf{D}$-form
\begin{equation}\label{eq:rho_def}
  \rho(\mathbf{w}) :=
  \frac{\langle\mathsf{S}^{(C)}\mathbf{w},\mathbf{w}\rangle_{\mathsf{D}}}
       {\|\mathbf{w}\|^2_{\mathsf{D}}}.
\end{equation}
Hypothesis~\eqref{eq:diss_on_range} asks precisely that $\rho(\mathbf{w})\notin(-\infty,0]$ for every nonzero $\mathbf{w}\in\operatorname{range}\mathsf{R}^{(C)}_{\mathrm{II}}$. We probe this by evaluating~\eqref{eq:rho_def} over a sample set $\mathcal{W}$ of $10^4$ random vectors
$\mathbf{w}=\mathsf{R}^{(C)}_{\mathrm{II}}\mathbf{x}$, with the entries of $\mathbf{x}$ taken from the standard (complex) normal distribution, and reporting the observed ranges
\begin{equation}\label{eq:rho_ranges}
\mathcal R_\rho = \left[\min_{\mathbf w\in \mathcal{W}}\real\rho(\mathbf{w}),\,\max_{\mathbf w\in \mathcal{W}}\real\rho(\mathbf{w})\right]\quad\text{and}\quad \mathcal I_\rho = \left[\min_{\mathbf w\in \mathcal{W}}\imag\rho(\mathbf{w}),\,\max_{\mathbf w\in\mathcal{W}}\imag\rho(\mathbf{w})\right].
\end{equation}

Clearly, this sampling can only exhibit values attained by $\rho$, so it may refute the hypothesis but not establish it. Either of two mechanisms keeps $\rho$ off the non-positive real axis: an imaginary part bounded away from zero, i.e.\ $0\notin\mathcal I_\rho$, which is the imaginary-part condition~\eqref{eq:diss_imag_part} of the energy balance in Section~\ref{sec:dissipativity_physical}; or a positive real part, $\mathcal R_\rho\subset(0,\infty)$. In the experiments reported in Figures~\ref{fig:rho_sampling_const}--\ref{fig:rho_sampling_air}, in all cases with $k=10$ and $kh_C=1$, at least one of the two holds in every configuration considered. Each figure shows $\mathcal R_\rho$ in its top row and $\mathcal I_\rho$ in its bottom row, across three mesh ratios $r_h\in\{0.5,1.0,1.5\}$.

Figure~\ref{fig:rho_sampling_const} considers a constant refractive index $n\equiv n_0$. As predicted by~\eqref{eq:variable_contrast_perturbation}, the sampled imaginary parts carry the sign of the contrast $m_0=1-n_0$: $\mathcal I_\rho$ lies entirely above zero for $n_0<1$ and entirely below it for $n_0>1$, and contracts towards zero as $n_0\to1$, where the contrast vanishes and the scatterer disappears. Near $n_0=1$ the imaginary part therefore does not separate $\rho$ from the set $(-\infty, 0]$ that is problematic for invertibility, yet in this case the real part does, with $\mathcal R_\rho$ clustering around $1$. 

The remaining two configurations consider the non-constant refractive index
\begin{equation}\label{eq:rho_variable_index}
  n(x) = n_0\Bigl(1+\delta\bigl(\exp(-10((x_1-0.5)^2+x_2^2))
                                   -\exp(-10((x_1+0.5)^2+x_2^2))\bigr)\Bigr),
\end{equation}
that perturbs a uniform background (either of glass or air, with $n_0\equiv1.5^{2}$ and $n_0\equiv1.0$ respectively) by two Gaussian bumps of opposite sign and amplitude $\delta\in[0,1]$. For glass (Figure~\ref{fig:rho_sampling_glass}) the background contrast dominates and the sampled imaginary parts stay bounded away from zero, $0\notin\mathcal I_\rho$, uniformly in $\delta$ and in the mesh ratio, as the sufficient condition~\eqref{eq:diss_imag_part} would require. For air (Figure~\ref{fig:rho_sampling_air}) the contrast is due entirely to the bumps and therefore changes sign, and so does the imaginary part of $\rho$: here $0\in\mathcal I_\rho$ and the sufficient condition~\eqref{eq:diss_imag_part} fails. The real part nonetheless remains positive and clustered near $1$, leaving the weaker hypothesis~\eqref{eq:diss_on_range}, the one actually assumed in Theorem~\ref{thm:P_invertible_unconditional}, satisfied on the whole sample. In none of the experiment configurations did $\rho$ approach the problematic set $(-\infty, 0]$, although this is not assured from the energy balance of Section~\ref{sec:dissipativity_physical}.

\begin{figure}[htbp]
    \centering
    \includegraphics[width=0.3\linewidth]{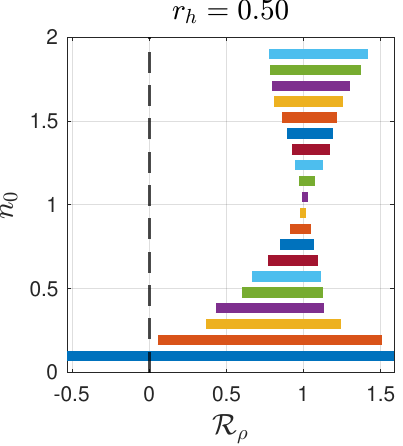}\hfill
    \includegraphics[width=0.3\linewidth]{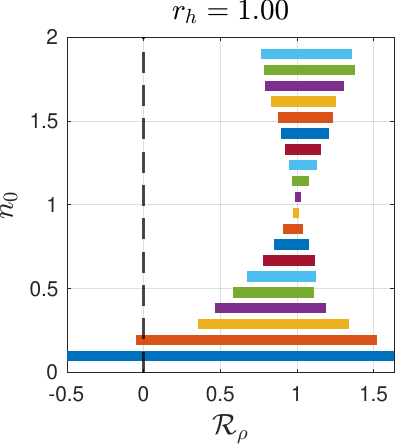}\hfill
    \includegraphics[width=0.3\linewidth]{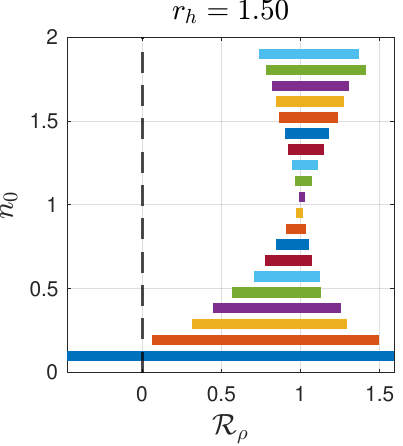}\\[6pt]
    \includegraphics[width=0.3\linewidth]{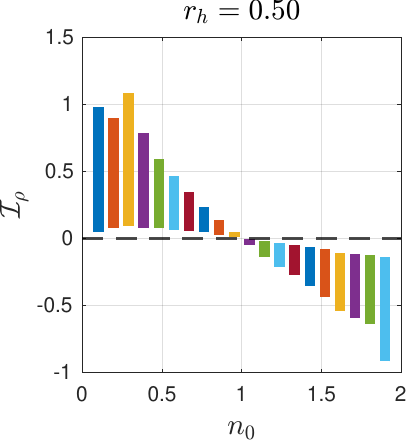}\hfill
    \includegraphics[width=0.3\linewidth]{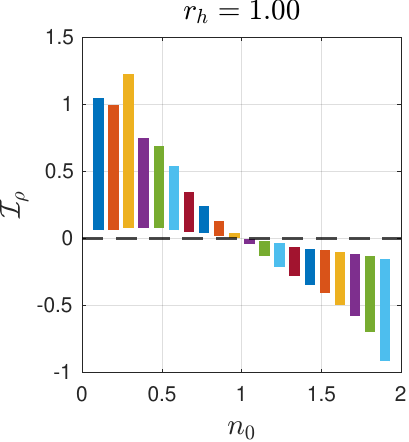}\hfill
    \includegraphics[width=0.3\linewidth]{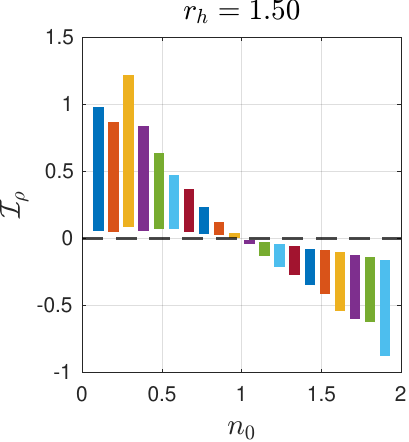}
    \caption{Constant refractive index $n\equiv n_0$: observed ranges~\eqref{eq:rho_ranges} of the normalized form~\eqref{eq:rho_def} over $10^4$ random $\mathbf{w}\in\operatorname{range} \mathsf{R}^{(C)}_{\mathrm{II}}$, as a function of $n_0$, at mesh ratios $r_h\in\{0.5,1.0,1.5\}$ (left to right), in all cases with $k=10$ and $k\,h_C=1$. Top row: real range $\mathcal R_\rho$, which stays positive and clustered near $1$ except at the smallest values of $n_0$. Bottom row: imaginary range $\mathcal I_\rho$, lying entirely above zero for $n_0<1$ and entirely below it for $n_0>1$, and contracting as $n_0\to1$, where the contrast $m_0=1-n_0$ vanishes; the sign of $\imag\rho$ is thus that of $m_0$, as predicted by~\eqref{eq:variable_contrast_perturbation}. Where one range approaches zero the other stays away from it, so $\rho$ avoids the non-positive real axis throughout.}
    \label{fig:rho_sampling_const}
\end{figure}

\begin{figure}[htbp]
    \centering
    \includegraphics[width=0.3\linewidth]{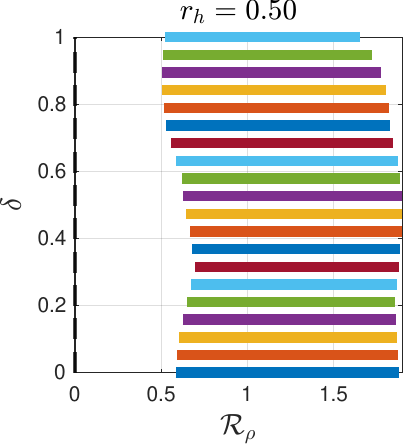}\hfill
    \includegraphics[width=0.3\linewidth]{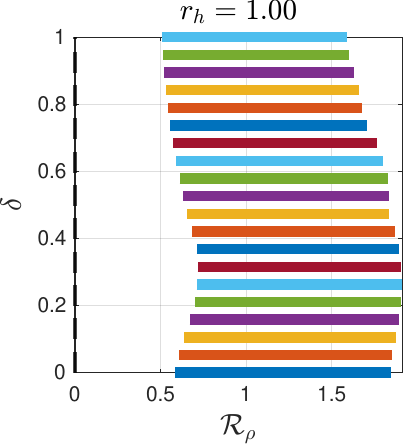}\hfill
    \includegraphics[width=0.3\linewidth]{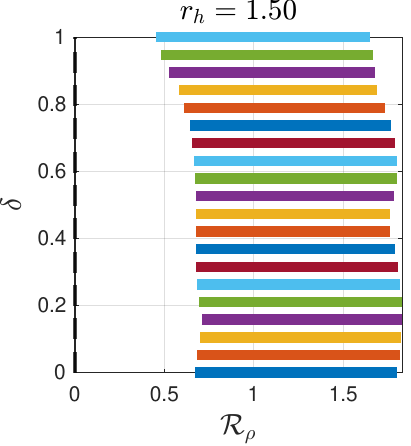}\\[6pt]
    \includegraphics[width=0.3\linewidth]{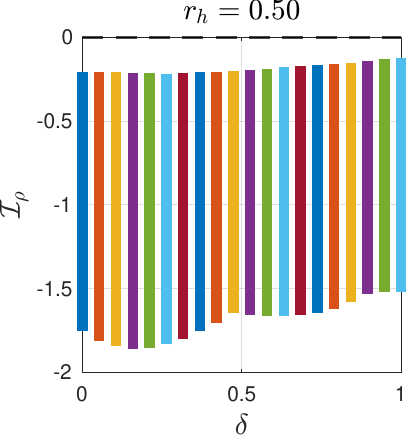}\hfill
    \includegraphics[width=0.3\linewidth]{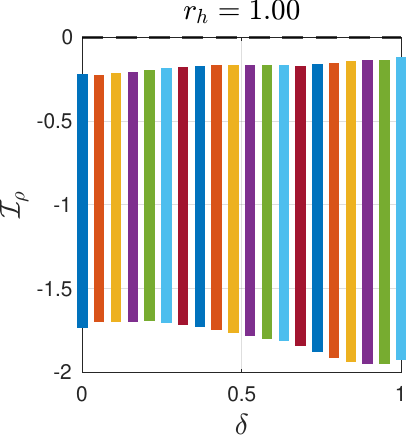}\hfill
    \includegraphics[width=0.3\linewidth]{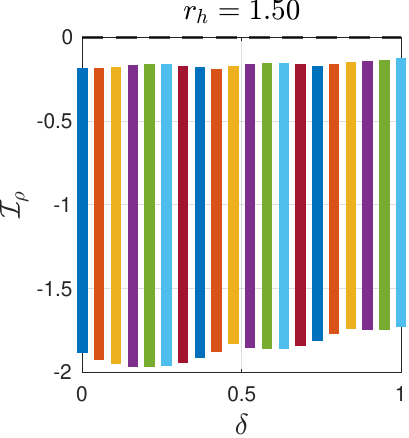}
    \caption{Bump-perturbed glass, i.e.\ the index~\eqref{eq:rho_variable_index} with $n_0=1.5^{2}$: observed ranges~\eqref{eq:rho_ranges} of the normalized form~\eqref{eq:rho_def} over $10^4$ random $\mathbf{w}\in\operatorname{range} \mathsf{R}^{(C)}_{\mathrm{II}}$, as a function of the perturbation amplitude $\delta$, at mesh ratios $r_h\in\{0.5,1.0,1.5\}$ (left to right), in all cases with $k=10$ and $k\,h_C=1$. Top row: real range $\mathcal R_\rho$, positive throughout. Bottom row: imaginary range $\mathcal I_\rho$, which carries the negative sign of the glass background contrast and remains bounded away from zero for every $\delta$ and every mesh ratio, so the sampling is consistent with the imaginary-part condition~\eqref{eq:diss_imag_part}.}
    \label{fig:rho_sampling_glass}
\end{figure}

\begin{figure}[htbp]
    \centering
    \includegraphics[width=0.3\linewidth]{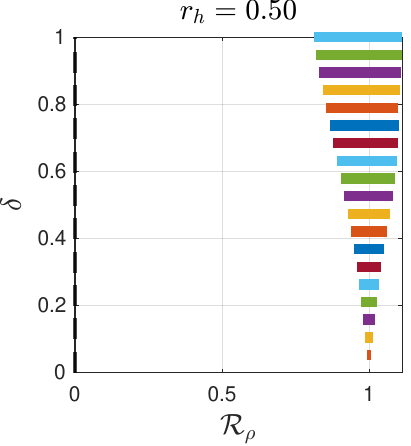}\hfill
    \includegraphics[width=0.3\linewidth]{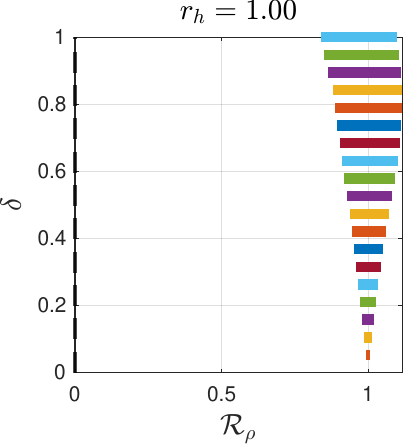}\hfill
    \includegraphics[width=0.3\linewidth]{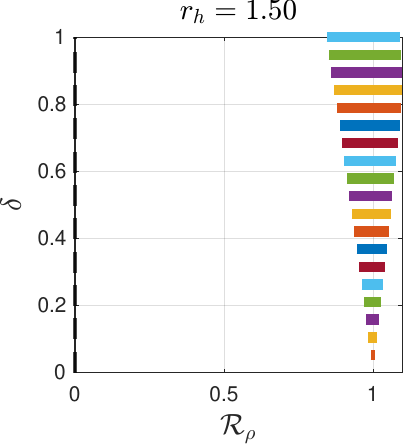}\\[6pt]
    \includegraphics[width=0.3\linewidth]{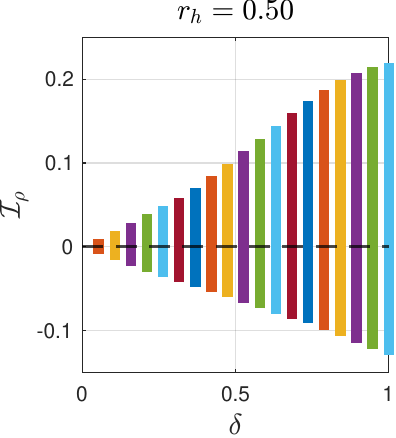}\hfill
    \includegraphics[width=0.3\linewidth]{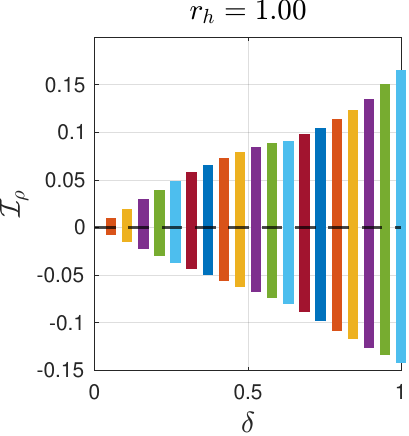}\hfill
    \includegraphics[width=0.3\linewidth]{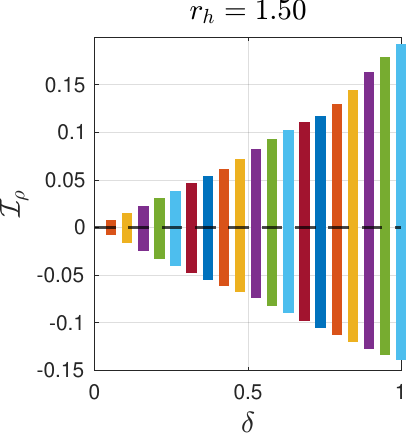}
    \caption{Bump-perturbed air, i.e.\ the index~\eqref{eq:rho_variable_index} with $n_0=1$: observed ranges~\eqref{eq:rho_ranges} of the normalized form~\eqref{eq:rho_def} over $10^4$ random $\mathbf{w}\in\operatorname{range} \mathsf{R}^{(C)}_{\mathrm{II}}$, as a function of the perturbation amplitude $\delta$, at mesh ratios $r_h\in\{0.5,1.0,1.5\}$ (left to right), in all cases with $k=10$ and $k\,h_C=1$. Here the background contrast vanishes and the bumps contribute contrast of both signs, so the imaginary range $\mathcal I_\rho$ (bottom row) contains zero and widens with $\delta$, and the imaginary-part condition~\eqref{eq:diss_imag_part} fails. The real range $\mathcal R_\rho$ (top row) nonetheless stays positive and clustered near $1$, so $\rho$ still avoids the non-positive real axis and the hypothesis~\eqref{eq:diss_on_range} of Theorem~\ref{thm:P_invertible_unconditional} continues to hold.}
    \label{fig:rho_sampling_air}
\end{figure}

\newpage
\bibliographystyle{abbrv}
\bibliography{References}
\end{document}